\documentclass[11pt,reqno]{amsart}
\usepackage{amsmath,amssymb, amscd, color, enumitem, cite}
\usepackage{mathrsfs, bm}
\usepackage{graphicx}
\usepackage{tikz}
\usepackage{pgfplots}
\usepackage{pgfplotstable}
\usepackage{tkz-fct}
\usetikzlibrary{pgfplots.polar}
\pgfplotsset{compat=newest}
\usepgfplotslibrary{fillbetween}
\usetikzlibrary{patterns}

\numberwithin{equation}{section}
\newtheorem{theorem}{Theorem}[section]
\newtheorem{lemma}[theorem]{Lemma}
\newtheorem{corollary}[theorem]{Corollary}

\theoremstyle{definition}
\newtheorem{definition}[theorem]{Definition}
\newtheorem{example}[theorem]{Example}
\theoremstyle{remark}
\newtheorem{remark}[theorem]{Remark}

\def\XXint#1#2#3{{\setbox0=\hbox{$#1{#2#3}{\int}$ }
\vcenter{\hbox{$#2#3$ }}\kern-.6\wd0}}

\definecolor{darkgreen}{rgb}{0.0,0.5,0.0}

\theoremstyle{definition}

\begin{document}

\title[Luxemburg Norm Localisation for Nonlocal Differential Equations]{Luxemburg Norm Localisation for Nonlocal Differential Equations in Variable Exponent Lebesgue Spaces}

\author[C. S. Goodrich]{Christopher S. Goodrich}
\address{School of Mathematics and Statistics\\
UNSW Sydney\\
Sydney, NSW 2052 Australia}
\email[Christopher S. Goodrich]{c.goodrich@unsw.edu.au}
\author[G. Nakhl]{Gabriel Nakhl}
\address{School of Mathematics and Statistics\\
UNSW Sydney\\
Sydney, NSW 2052 Australia}
\email[Gabriel Nakhl]{g.nakhl@student.unsw.edu.au}
\keywords{Nonlocal differential equation; Luxemburg norm; variable growth; positive solution; convolution.}
\subjclass[2020]{Primary: 34B10, 34B18, 42A85, 44A35, 46E30.  Secondary: 26A33, 47H30.}


\begin{abstract}
We investigate a class of variable growth nonlocal differential equations of Kirchhoff-type having the general form
\begin{equation}
-A\!\left(\int_0^1 b(1-s)\big(u(s)\big)^{p(s)}\,ds\right) u''(t) = \lambda f\big(t,u(t)\big), \quad t\in(0,1),\notag
\end{equation}
where $A$ is a possibly sign-changing function.  Our analysis is carried out in the variable-exponent Lebesgue space $L^{p(\cdot)}([0,1])$ under the standing hypothesis $p(t)>1$. We demonstrate that using the Luxemburg norm allows for a much sharper localisation of the solution to the nonlocal problem.  Moreover, conditions imposed on both $\lambda$ and $f$ are appreciably weakened by analysing the problem within the Luxemburg norm framework.  An example explicitly demonstrates both the qualitative and the quantitative advantages over earlier techniques.
\end{abstract}

\maketitle

\section{Introduction}

Variable-exponent Lebesgue spaces $L^{p(\cdot)}(\Omega)$ extend the classical theory by allowing the integrability exponent to vary measurably across the domain. They now play a routine role in a variety of areas of mathematics such as the analysis of PDEs \cite{cao_class_2021,edmunds_sobolev_2002,ge_small_2021-2,vetro_variable_2022}, mathematical modelling (e.g., elasticity theory, electrorheological fluids, and image enhancement)\cite{andreianov_numerical_2023,chen_mathematical_2006,rajagopal1,ruzicka_electrorheological_2000,zhikov1}, and regularity theory \cite{bronzi_regularity_2020,coscia_holder_2009,feyfoss2,ragusatachikawa1,tachikawa_singular_2014}.  The associated Luxemburg norm
\begin{equation*}
\|u\|_{L^{p(\cdot)}}=\inf\left\{\delta>0:\int_{\Omega}\left|\frac{u(x)}{\delta}\right|^{p(x)}\,dx\le 1\right\}
\end{equation*}
is the natural quantitative device for working in these spaces \cite{diening_lebesgue_2011,khamsi_fixed_2015} since it explicitly incorporates the variable exponent $p$ in its measurement.

Consider now the finite-convolution functional defined by
\begin{equation}
(u*v)(t):=\int_0^tu(t-s)v(s)\ ds\text{, }0\le t\le 1,\notag
\end{equation}
for $u$, $v\in L^1\big((0,1)\big)$.  For $\lambda>0$ a real parameter, in what follows we study positive solutions of the nonlocal differential equation
\begin{equation}\label{eq1.1}
-A\left(\left(b*u^{p(\cdot)}\right)(1)\right)\,u''(t)=\lambda\,f\bigl(t,u(t)\bigr),\qquad t\in(0,1),
\end{equation}
posed in $L^{p(\cdot)}([0,1])$, where \eqref{eq1.1} will be equipped with some boundary data such as, for example, the Dirichlet data $u(0)=0=u(1)$. Throughout we assume $A\ :\ [0,+\infty)\to\mathbb{R}$, $f\ :\ [0,1]\times[0,+\infty)\to[0,+\infty)$, and $p\ :\ [0,1]\to(1,+\infty)$ are continuous, and $b\in L^{1}((0,1])$ is a.e.\ nonnegative -- see Section 2 for the precise assumptions.  Note that the nonlocal element in \eqref{eq1.1} is
\begin{equation}
\left(b*u^{p(\cdot)}\right)(1)=\int_0^1b(1-s)\big(u(s)\big)^{p(s)}\ ds,\notag
\end{equation}
which, therefore, incorporates a \emph{variable} growth element.  As a reference point, when $b(t)\equiv 1$ and $p(t)\equiv p_0>1$ equation \eqref{eq1.1} reduces to
\begin{equation}
-A\bigl(\|u\|_{L^{p_0}}^{p_0}\bigr)u''(t)=\lambda f\bigl(t,u(t)\bigr),\notag
\end{equation}
which is evidently related to the one-dimensional steady-state analogue of the classical Kirchhoff model
\begin{equation}
u_{tt}-A\bigl(\|Du\|_{L^2}^{2}\bigr)\Delta u=\lambda f\big(\bm{x},u(\bm{x})\big).\notag
\end{equation}
Note, further, that the convolutional structure allows one to admit a variety of nonlocalities such as a Riemann-Liouville fractional integral by selecting $\displaystyle b(t)=\frac{1}{\Gamma(\alpha)}t^{\alpha-1}$, $0<\alpha<1$ -- see Goodrich and Lizama \cite{goodrichlizama2}, Lan \cite{lan2,lan_compactness_2020}, Podlubny \cite{podlubny_fractional_1999}, and Webb \cite{webb0}.

When it comes to ordinary differential equations with variable growth, the results seem to be rather scarce.  Very recently Garc\'{i}a-Huidobro, et al. \cite{garcia1,garcia2} considered an analysis of systems of \emph{\underline{local}} ordinary differential equations within the framework of variable exponents.  On the other hand, also recently both the first author \cite{goodrich_px-growth_2025,goodrich27} and the second author \cite{goodrich2025unified} have made some attempts to analyse \emph{\underline{nonlocal}} ordinary differential equations with $p(x)$-type growth in the nonlocality -- i.e., as in \eqref{eq1.1}.  However, none of these papers considers the possibility of leveraging the Luxemburg norm to obtain more natural and better control within the context of nonlocal, variable growth problems.

To expand on the point emphasised in the previous paragraph, consider the supremum-norm cone
\begin{equation}
\mathscr{K}_{\infty}:=\left\{ u \in \mathscr{C}([0,1]) : u \ge 0, \ \min_{t \in [\alpha,\beta]} u(t) \ge \eta_0 \|u\|_\infty, \ (\bm{1}*u)(1) \ge C_0 \|u\|_\infty \right\},\notag
\end{equation}
where $\bm{1}$ denotes the constant map $\bm{1}\ : \ \mathbb{R}\rightarrow\{1\}$, both $0<\eta_0\le1$ and $0<C_0<1$ are constants, and $0\le\alpha<\beta\le 1$ are constants.  Each of the recent papers \cite{goodrich_px-growth_2025,goodrich27,goodrich2025unified} utilised this $\mathscr{K}_{\infty}$ cone; again, we note that \cite{garcia1,garcia2} did not consider nonlocal problems at all.  Some weaknesses of the cone $\mathscr{K}_{\infty}$ are the following.
\begin{itemize}
\item $L^\infty$ estimates do not provide precise control in variable exponent spaces.

\item The cone is not intrinsically adapted to $L^{p(\cdot)}$ growth, which may be restrictive for variable exponent problems.
\end{itemize}
A natural question is then: could we modify $\mathscr{K}_{\infty}$ by means of the Luxemburg norm to analyse more naturally solutions of \eqref{eq1.1} within a topological fixed point milieu?

With this question in mind, then, the transition we make in this paper is to base the analysis on Luxemburg-modular bounds. Because the Luxemburg norm is intrinsic to $L^{p(\cdot)}$, it produces sharper and more natural estimates than those obtainable through $L^\infty$, while still allowing a fixed-point scheme via a carefully constructed cone.  In particular, in this work we introduce the following hybrid-type cone.
\[
\mathscr{K}_{\text{hybrid}} := \left\{ u \in \mathscr{C}([0,1])\ : \
\begin{aligned}
& u \ge 0, \\
& \min_{t \in [\alpha,\beta]} u(t) \ge \eta_0 \|u\|_\infty, \\
& (\bm{1}*u)(1) \ge C_0 \|u\|_{L^{p(\cdot)}}
\end{aligned}
\right\}
\]
Note that the principal dissimilarity between $\mathscr{K}_{\infty}$ and $\mathscr{K}_{\text{hybrid}}$ is that the latter replaces an $L^{\infty}$-type coercivity relation, namely $(\bm{1}*u)(1)\ge C_0\Vert u\Vert_{\infty}$, with a Luxemburg-norm coercivity relation, namely $(\bm{1}*u)(1)\ge C_0\Vert u\Vert_{L^{p(\cdot)}}$.

Whilst this dissimilarity might seem to border on trivial, as we demonstrate explicitly in this paper (cf., Remark \ref{remark2.22aaa} and Example \ref{example2.12}, for example) its salutary effects are surprisingly substantial.  Indeed, we obtain bounds that are both quantitatively sharper and structurally more natural for variable exponents, avoiding artefacts introduced by forcing the analysis through $L^\infty$.  For example, in Example \ref{example2.12} we demonstrate that for a specific problem the use of $\mathscr{K}_{\infty}$ leads to existence of solution under the conditions
\begin{equation}\label{eq1.2ggg}
\lambda\left(\min_{\left[\frac{1}{4},\frac{3}{4}\right] \times \left[\frac{1}{200},\frac{102}{25}\sqrt{2}\right]}f(t,u)\right)\gtrapprox0.002\ \text{ and }\ \lambda\left(\max_{\left[0, 1\right]\times\left[0, 4\sqrt2(\sqrt3+1)\right]}f(t,u)\right)\lessapprox9965.848,
\end{equation}
whereas the use of $\mathscr{K}_{\text{hybrid}}$ leads to existence of solution under the (weaker) conditions
\begin{equation}\label{eq1.3ggg}
\lambda\left(\min_{\left[\frac{1}{4},\frac{3}{4}\right] \times \left[\frac{1}{200},4\sqrt{2}\,(1/2500)^{1/5}\right]}f(t,u)\right)\gtrapprox1.869\times10^{-4}\ \text{ and } \ \lambda\left(\max_{[0,1]\times\left[0,\, 4\sqrt{6}\right]}f(t,u)\right)\lessapprox1.21\times 10^4.
\end{equation}
Furthermore, the localisation of the solution, say $u_0$, via $\mathscr{K}_{\infty}$ versus $\mathscr{K}_{\text{hybrid}}$ is
\begin{equation}\label{eq1.4ggg}
0.02<\Vert u_0\Vert_{\infty}\lessapprox15.455\ \text{ versus } \ 0.02<\Vert u_0\Vert_{L^{p(\cdot)}}\lessapprox1.732,
\end{equation}
respectively.  If in inequalities \eqref{eq1.2ggg}--\eqref{eq1.3ggg} we use the identification (in bold)
\begin{equation}
\begin{split}
\lambda\left(\min_{\left[\frac{1}{4},\frac{3}{4}\right]\times\left[\frac{1}{200},\textbf{ Inner Height}\right]}f(t,u)\right)&\gtrapprox\textbf{Inner Radius}\\
\lambda\left(\max_{[0,1]\times\left[0,\textbf{ Outer Height}\right]}f(t,u)\right)&\lessapprox\textbf{Outer Radius},\notag
\end{split}
\end{equation}
then Figures \ref{fig1}--\ref{fig2} visually demonstrate the degree of improvement of our theory.  Observe, that in some cases our improvements are of an order of magnitude or greater -- or, alternatively, on a percentage-wise basis as much as $92\%$.  Consequently, the improvements afforded by our new methodology are not merely \textquotedblleft$\varepsilon$-improvements\textquotedblright, as it were.  They are significant.
\begin{figure}
\begin{minipage}[c]{0.495\linewidth}
\includegraphics[width=\linewidth]{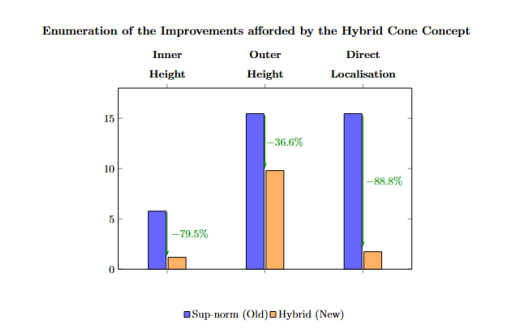}
\caption{A comparison of the outer/inner height in \eqref{eq1.2ggg}--\eqref{eq1.3ggg} and direct localisation in \eqref{eq1.4ggg}.}\label{fig1}
\end{minipage}
\hfill
\begin{minipage}[c]{0.495\linewidth}
\includegraphics[width=\linewidth]{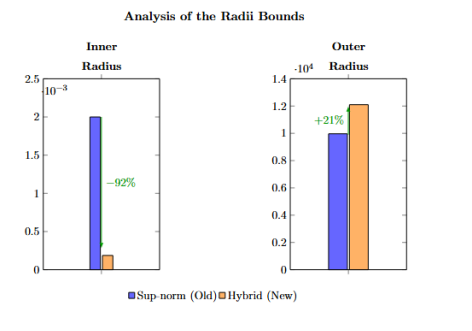}
\caption{A comparison of the outer/inner radius in \eqref{eq1.2ggg}--\eqref{eq1.3ggg}.}\label{fig2}
\end{minipage}
\end{figure}

As part of providing a more natural framework in which to study nonlocal variable exponent problems, we establish new embedding theorems linking $L^{p(\cdot)}$ to related function spaces, which supply the compactness needed for the fixed-point scheme within the hybrid cone $\mathscr{K}_{\text{hybrid}}$. Besides their role here, these embeddings enrich the functional analysis of variable exponent spaces and may be of independent interest to practitioners in the area.

Having characterised the novelty of the Luxemburg-focused aspect of our work here, let us conclude by mentioning another important contribution of our methodology.  In general, relatively severe restrictions are imposed on the nonlocal coefficient $A$ in nonlocal differential equations (whether PDEs or ODEs).  Far and away the most common \cite{alves_existence_2015,azzouz_existence_2012,bellamouchi1,biagi1,biagi2,boulaaras0,boulaaras1,calamai0,calamai1,correa1,correa2,do4,graef2,infante0,infante2,infante99,li1,stanczy1,wang1} is that $A(t)>0$ for all $t\ge0$; on occasion this has been coupled with a monotonicity-type condition \cite{yanma1,yan1}.  Relatively recently, somewhat less restrictive assumptions have been imposed, such as assuming \cite{ambrosetti1} that $A$ can vanish at $+\infty$, that $A$ vanishes at a single specified point \cite{delgado_non-local_2020}, or that $A(t)$ is positive on some neighbourhood of zero \cite{junior_positive_2018-1}.  Although the restrictions imposed by \cite{ambrosetti1,delgado_non-local_2020,junior_positive_2018-1} are definitely much less restrictive, they still place on $A$ very specific limitations.

Beginning with the first author's work \cite{goodrich_topological_2021-1}, which was subsequently generalised by Chu and Hao \cite{chu_positive_2025}, Goodrich, et al. \cite{goodrich_topological_2021-2,goodrich_nonlocal_2021,goodrich20,goodrich_application_2024-1,goodrich_existence_2022}, Hao and Wang \cite{hao1}, Shibata \cite{shibata3}, and Song and Hao \cite{song1}, a new methodology has been introduced that allows for the relaxation of the conditions enumerated in the preceding paragraph.  In particular, this new methodology relies on using specialised order cones and attendant open sets in order to obtain precise information on where the argument of the nonlocal element lives.  For example, in the case of problem \eqref{eq1.1}, we use the open set, for $\rho>0$,
\begin{equation}\label{eq1.5ggg}
\widehat{V}_{\rho}:=\left\{u\in\mathscr{K}_{\text{hybrid}}\ : \ \left(b*u^{p(\cdot)}\right)(1)<\rho\right\}.
\end{equation}
The key fact regarding \eqref{eq1.5ggg} is that whenever $u\in\partial\widehat{V}_{\rho}$ it follows that $\displaystyle\left(b*u^{p(\cdot)}\right)(1)=\rho$.  Since our arguments (see Theorem \ref{theorem2.25}) use topological fixed point theory, we work within sets of the form $\partial\widehat{V}_{\rho}$.  Consequently, we are able to localise precisely where $A$ must be positive, and so, this allows us to eliminate the usual restrictions identified earlier.  All in all, our results demonstrate that we can maintain the less restrictive hypotheses introduced in \cite{goodrich_topological_2021-1}, together with the new cone $\mathscr{K}_{\text{hybrid}}$, all whilst broadening the applicability of the existence theory.

\section{Preliminaries, Operator Definitions, and Existence Theory}\label{sec:preliminaries}

\subsection{General Assumptions}

We begin by setting the notation and spaces used throughout this paper. Let $\mathscr{C}([0,1])$ denote the space of continuous functions on $[0,1]$, equipped with the supremum norm $\Vert\cdot\Vert_\infty$, making it a Banach space.  As in, for example, \cite{goodrich2025unified}, by $\bm{1}$ we denote the constant function $\bm{1}:\mathbb{R}\to\{1\}$, and likewise by $\bm{0}$ we denote the constant function $\bm{0}:\mathbb{R}\to\{0\}$. For a measurable set $E\subset[0,1]$, we write $\bm{1}_E$ for its indicator: $\bm{1}_E(t)=1$ if $t\in E$ and $0$ otherwise. By $*$ we denote the finite convolution functional on $[0,1]$ so that
\begin{equation*}
(u*v)(t):=\int_0^t u(t-s)v(s)\,ds,\quad 0\le t\le 1,
\end{equation*}
for $u, v\in L^1\big((0,1)\big)$. Finally, for a given continuous function $h:[0,1]\times[0,+\infty)\to[0,+\infty)$ and real numbers $0\le a<b\le1$ and $0\le c<d<+\infty$ we denote by $h_{[a,b]\times[c,d]}^{m}$ and $h_{[a,b]\times[c,d]}^{M}$ the following quantities.
\begin{equation*}
\begin{split}
h_{[a,b]\times[c,d]}^{m}&:=\min_{(t,u)\in[a,b]\times[c,d]}h(t,u)\\
h_{[a,b]\times[c,d]}^{M}&:=\max_{(t,u)\in[a,b]\times[c,d]}h(t,u)
\end{split}
\end{equation*}

Since we will study solutions of (1.1) subject to given boundary data, we will require the notion of a Green's function $G\ : \ [0,1]\times[0,1]\rightarrow[0,+\infty)$.  The following are two, commonly occurring examples of Green's functions, together with the boundary data they encode -- cf., \cite{erbe_positive_2004}.
\begin{align}
G(t,s) &:=
\begin{cases}
t(1-s), & 0 \le t \le s \le 1, \\
s(1-t), & 0 \le s \le t \le 1,
\end{cases}
&& \text{(Dirichlet)} \label{eq:G_dirichlet} \\[1ex]
G(t,s) &:=
\begin{cases}
t, & 0 \le t \le s \le 1, \\
s, & 0 \le s \le t \le 1,
\end{cases}
&& \text{(Right-focal)} \label{eq:G_rightfocal}
\end{align}

We now state the general assumptions underlying our study of \eqref{eq1.1}.  Both \eqref{eq:G_dirichlet} and \eqref{eq:G_rightfocal} satisfy (H3), along with many other standard choices (see \cite{goodrich_existence_2010,infante_existence_2014,goodrich_new_2018}).

\begin{list}{}{\setlength{\leftmargin}{.5in}\setlength{\rightmargin}{0in}}
\item[\textbf{H1:}] The functions $A:[0,+\infty)\to\mathbb{R}$, $f:[0,1]\times[0,+\infty)\to[0,+\infty)$, and $b:(0,1]\to[0,+\infty)$ satisfy:
\begin{enumerate}
    \item $A$ and $f$ are continuous on their domains.
    \item $b\in L^1((0,1]) \text{ with } b(t)>0 \text{ a.e. }t\in(0,1]$
    \item There exist $0<\rho_1<\rho_2$ such that $A(t) > 0$ for all $t \in [\rho_1, \rho_2]$.
\end{enumerate}

\item[\textbf{H2:}] The exponent $p:[0,1]\to(1,+\infty)$ is continuous and satisfies
\[
1 <  \min_{t\in[0,1]} p(t)=:p^- \le p(t) \le p^+ := \max_{t\in[0,1]} p(t) < +\infty.
\]

\item[\textbf{H3:}] The Green's function $G:[0,1]\times[0,1]\to[0,+\infty)$ satisfies each of the following.
\begin{enumerate}
    \item Let $\displaystyle\mathscr{G}(s) := \max_{t\in[0,1]} G(t,s)$. Assume $\mathscr{G}(s)>0$ for all $s\in(0,1)$. There exist $0\le \alpha < \beta \le 1$ and $\eta_0 = \eta_0(\alpha,\beta)\in(0,1]$ such that
    \begin{equation*}
    \min_{t\in[\alpha,\beta]} G(t,s) \ge \eta_0 \mathscr{G}(s), \quad \forall s\in[0,1].
    \end{equation*}
    \item Define \begin{equation*}\displaystyle
    C_0 := \inf_{s\in(0,1)} \frac{1}{\mathscr{G}(s)} \int_0^1 G(t,s)\, dt.
    \end{equation*}
    Then $1>C_0>0$.
\end{enumerate}
\end{list}

Recall, furthermore, the variable exponent Lebesgue space $L^{p(\cdot)}([0,1])$, a Banach function space equipped with the Luxemburg norm (see, e.g., \cite{diening_lebesgue_2011}):
\[
\Vert u \Vert_{L^{p(\cdot)}} := \inf \left\{ \delta > 0 : \int_0^1 \left| \frac{u(x)}{\delta} \right|^{p(x)} \, dx \le 1 \right\}.
\]
Henceforth, for $u \in L^{p(\cdot)}([0,1])$ we refer to the functional
\[
I_{p(\cdot)}(u) := \int_0^1 \big|u(t)\big|^{p(t)}\,dt
\]
as the \emph{modular}, terminology which is standard in this area of study.  The following standard estimates link the Luxemburg norm and its associated modular.

\begin{lemma}\label{lem:Luxemburg_basic}
For $u \in L^{p(\cdot)}([0,1])$ it holds that $\displaystyle I_{p(\cdot)}\!\left(\frac{u}{\|u\|_{L^{p(\cdot)}}}\right)=1$ whenever $u\not\equiv 0$.
\end{lemma}

\begin{corollary}\label{cor:mod_norm_ge_1}
Assume condition \textnormal{(H2)}. For any $u\in L^{p(\cdot)}([0,1])$ with $\|u\|_{L^{p(\cdot)}}\ge 1$,
\[
\|u\|_{L^{p(\cdot)}}^{\,p^-} \le \int_0^1 |u(t)|^{p(t)}\,dt 
\le \|u\|_{L^{p(\cdot)}}^{\,p^+}.
\]
\end{corollary}

\begin{proof}
Standard; see, e.g., \cite{diening_lebesgue_2011}.
\end{proof}

\begin{corollary}\label{cor:mod_norm_le_1}
Assume condition \textnormal{(H2)}.  For any $u\in L^{p(\cdot)}([0,1])$ with $0<\|u\|_{L^{p(\cdot)}}\le 1$,
\[
\|u\|_{L^{p(\cdot)}}^{\,p^+} \le \int_0^1 |u(t)|^{p(t)}\,dt 
\le \|u\|_{L^{p(\cdot)}}^{\,p^-}.
\]
\end{corollary}

\begin{proof}
This follows directly from the definition of the Luxemburg norm and the bounds $p^-\le p(t)\le p^+$ a.e.; see also \cite{diening_lebesgue_2011}.
\end{proof}

\subsection{The Hybrid Functional Analytic Framework}

As mentioned in Section 1, the cone we will use throughout this paper is $\mathscr{K}_{\text{hybrid}}$, which, we recall, is defined as follows.
\[
\mathscr{K}_{\text{hybrid}} := \left\{ u \in \mathscr{C}([0,1])\ : \ 
\begin{aligned}
& u \ge 0, \\
& \min_{t \in [\alpha,\beta]} u(t) \ge \eta_0 \|u\|_\infty, \\
& (\bm{1}*u)(1) \ge C_0 \|u\|_{L^{p(\cdot)}}
\end{aligned}
\right\}
\]
Attendant to this cone, and also as mentioned in Section 1, for a real number $\rho>0$ we will use the (relatively) open set
\begin{equation}
\widehat{V}_{\rho}:=\left\{u\in\mathscr{K}_{\text{hybrid}}\ : \ \left(b*u^{p(\cdot)}\right)(1)<\rho\right\}.\notag
\end{equation}
By the continuity of $\displaystyle\Phi(u):=\left(b*u^{p(\cdot)}\right)(1)$ in the $\|\cdot\|_\infty$–topology, 
the set $\widehat V_\rho=\{u\in\mathscr{K}_{\mathrm{hybrid}}:\Phi(u)<\rho\}$ is relatively open, and
\[
\partial\widehat V_\rho=\{u\in\mathscr{K}_{\mathrm{hybrid}}:\Phi(u)=\rho\};
\]
see, for example, \cite[pp. 5--6]{chu_positive_2025}. We will use the characterisation of $\partial\widehat{V}_{\rho}$ frequently in what follows.

Our first lemma in this subsection is a lower bound on $\Vert u\Vert_{\infty}$ provided that $u\in\partial\widehat{V}_{\rho}$ for some $\rho>0$.  This result may be found in \cite[Lemma 2.3]{goodrich2025unified}.

\begin{lemma}\label{lem:sup-lower-bound}
Suppose that $u\in\partial\widehat{V}_{\rho}$ for some $\rho>0$.
Then
$$\Vert u\Vert_{\infty}\ge\left(\frac{\rho}{(b*\bm{1})(1)}\right)^{\frac{1}{p^+}}+\varepsilon_1(\rho,b),$$
\noindent
where
$$\varepsilon_1(\rho,b):=\begin{cases}
\left(\frac{\rho}{(b*\bm{1})(1)}\right)^{\frac{1}{p^-}}-\left(\frac{\rho}{(b*\bm{1})(1)}\right)^{\frac{1}{p^+}}\text{, }&0<\left(\frac{\rho}{(b*\bm{1})(1)}\right)<1\\
0\text{, }&\left(\frac{\rho}{(b*\bm{1})(1)}\right)\ge1
\end{cases}.$$
\end{lemma}

\begin{proof}
Omitted.
\end{proof}

We next introduce a definition, which captures the principal novelty of the hybrid cone, namely stating the coercivity condition in terms of $\Vert\cdot\Vert_{L^{p(\cdot)}}$.

\begin{definition}\label{def:hybrid-thickness}
Let $\alpha$, $\beta$, $\eta_0$, and $C_0$ be the constants from \emph{(H3)}.
A function $u\in \mathscr{C}([0,1])\cap L^{p(\cdot)}([0,1])$ satisfies the \emph{hybrid thickness condition} if both
\begin{equation*}
\min_{t\in[\alpha,\beta]} u(t) \ge \eta_0\,\|u\|_\infty,
\end{equation*}
and
\begin{equation*}\label{eq:hyb-Y}
(\bm{1}*u)(1) \ge C_0\,\|u\|_{L^{p(\cdot)}}
\end{equation*}
hold.
\end{definition}


As we argue in our next collection of results, the hybrid cone $\mathscr{K}_{\mathrm{hybrid}}$ is, in fact, an enlargement of the original cone, $\mathscr{K}_{\infty}$, in~\cite{goodrich2025unified}.

\begin{lemma}\label{lem:Linf_to_Lp_embedding}
Assume condition \textnormal{(H2)}. Then $\mathscr{C}([0,1]) \subset L^{p(\cdot)}([0,1])$ and 
$\|u\|_{L^{p(\cdot)}} \le \|u\|_\infty$ for all $u\in \mathscr{C}([0,1])$.
\end{lemma}

\begin{proof}
Let $u\in\mathscr{C}([0,1])$.  If $\|u\|_\infty=0$, then $u\equiv 0$ and the desired claim is trivially true. Otherwise, set $M:=\|u\|_\infty>0$. Since $|u(x)|/M\le 1$ for all $x\in[0,1]$, we have $\big(|u(x)|/M\big)^{p(x)}\le 1$. Integrating yields
\[
\int_0^1 \left(\frac{|u(x)|}{M}\right)^{p(x)}\,dx \le 1.
\]
By the definition of the Luxemburg norm it follows that $\|u\|_{L^{p(\cdot)}}\le M=\|u\|_\infty$, so $\mathscr{C}([0,1])\subseteq L^{p(\cdot)}([0,1])$. The inclusion is proper, for if $E\subset[0,1]$ satisfies $0<|E|<1$, and we set $u:=\bm{1}_E$, then
\[
\int_0^1 |u(t)|^{p(t)}\,dt=\int_E 1\,dt=|E|<+\infty.
\]
So, $u\in L^{p(\cdot)}([0,1])$, whereas $u\notin \mathscr{C}([0,1])$.  And this completes the proof.
\end{proof}

\begin{lemma}\label{lemma2.8bbb}
Assume conditions {\rm(H2)}--{\rm(H3)}.  We have the inclusion
\[
\mathscr{K}_{\infty} \subseteq\mathscr{K}_{\mathrm{hybrid}}.
\]
\end{lemma}

\begin{proof}
Let $u \in \mathscr{K}_{\infty}$. Then
\[
(\bm{1}*u)(1) \ge C_0 \|u\|_\infty \ge C_0 \|u\|_{L^{p(\cdot)}},
\]
where the second inequality follows from Lemma~\ref{lem:Linf_to_Lp_embedding}. Thus, $u \in \mathscr{K}_{\mathrm{hybrid}}$.
\end{proof}

The next lemma demonstrates that the inclusion of Lemma \ref{lemma2.8bbb} may be strengthened to a strict inclusion.

\begin{lemma}
Assume conditions {\rm(H2)}--{\rm(H3)}.  Then
\[
\mathscr K_\infty \subsetneq \mathscr K_{\mathrm{hybrid}}.
\]
\end{lemma}

\begin{proof}
It is sufficient to consider the constant exponent case $p(x)\equiv p>1$. First, we choose the measure $m$ of our function's support. Since $p>1$ and $C_0 \in (0,1)$, we have $p/(p-1) > 1$, which implies $C_0^{p/(p-1)} < C_0$. This gap allows us to choose a number $0<m<1$ such that each of the following is true for our function $u$ to be constructed in the next paragraph.
\begin{itemize}
    \item $m < C_0$ (which will ensure $u \notin \mathscr{K}_\infty$)
    \item $m^{(p-1)/p} > C_0$ (which will ensure $u \in \mathscr{K}_{\mathrm{hybrid}}$)
    \item $m \ge \beta - \alpha$ (to house the plateau)
\end{itemize}
Specifically, we can choose $m$ such that $\max\left\{C_0^{p/(p-1)}, \beta-\alpha\right\} < m < C_0$.

Now, let $E \subseteq [0,1]$ be a measurable set such that $[\alpha, \beta] \subseteq E$ and $|E| = m$. Fix $\xi_0 > 0$. By a standard mollification, we can construct a continuous "bump" function $\phi \in \mathscr{C}([0,1])$ with the following properties:
\begin{itemize}
    \item $0 \le \phi(t) \le 1$ for all $t \in [0,1]$;
    \item $\phi(t) \equiv 1$ for all $t \in [\alpha, \beta]$;
    \item $\operatorname{supp}(\phi) \subseteq E$; and
    \item each of the $L^1$ and $L^p$ norms of $\phi$ is arbitrarily close to those of $\mathbf{1}_E$. That is, for any $\varepsilon > 0$, we can construct $\phi$ such that both $\displaystyle\int_0^1 \phi(t) \,dt > m - \varepsilon$ and $\|\phi\|_{L^p} < m^{1/p} + \varepsilon$.
\end{itemize}
Let $u(t) := \xi_0 \phi(t)$. We now verify this $u$ has the required properties.

We verify first that $u\in\mathscr{K}_{\text{hybrid}}$.  On the one hand, by construction, $u \in \mathscr{C}([0,1])$. Note that $\|u\|_\infty = \xi_0$. Since $\phi \equiv\bm{1}$ on $[\alpha, \beta]$, we have $\displaystyle\min_{t \in [\alpha, \beta]} u(t) = \xi_0$. Thus, the plateau condition $\displaystyle\min_{t \in [\alpha, \beta]} u(t) \ge \eta_0 \|u\|_\infty$ holds for any $\eta_0 \le 1$.  On the other hand, by our choice of $m$, we have that $m^{(p-1)/p} > C_0$, which implies $m > C_0 m^{1/p}$. Since $m - C_0 m^{1/p}>0$, we can choose our approximation $\phi$ to be sufficiently close to $\mathbf{1}_E$ (i.e., choose $\varepsilon > 0$ small enough) such that $m - \varepsilon \ge C_0 (m^{1/p} + \varepsilon)$. With such a $\phi$, we have that
    \[
    (\mathbf{1}*u)(1)=\int_0^1u(t)\ dt = \xi_0 \int_0^1 \phi(t) \,dt > \xi_0 (m - \varepsilon) \ge \xi_0 C_0 (m^{1/p} + \varepsilon) > C_0 \left( \xi_0 \|\phi\|_{L^p} \right) = C_0 \|u\|_{L^p}.
    \]
Thus, both conditions are met, and so, $u \in \mathscr{K}_{\mathrm{hybrid}}$, as desired.

We next show that $u\notin\mathscr{K}_{\infty}$.  As above, we have that $\displaystyle(\bm{1}*u)(1) = \xi_0 \int_0^1 \phi(t) \,dt$. Since $\operatorname{supp}(\phi) \subseteq E$ and $0\le\phi(t)\le 1$, it follows that $\displaystyle\int_0^1 \phi(t) \,dt \le \int_E 1 \,dt = |E| = m$.  But if the condition $(\bm{1}*u)(1)\ge C_0\Vert u\Vert_{\infty}$ was to be satisfied, then it would have to hold that $\xi_0 m \ge C_0 \xi_0$.  Since this would imply that $m \ge C_0$, which would be a contradiction to our choice of $m < C_0$, we conclude that $u \notin \mathscr{K}_\infty$.

All in all, therefore, we have constructed a function $u \in \mathscr{K}_{\mathrm{hybrid}} \setminus \mathscr{K}_\infty$. Therefore, the inclusion is strict, as claimed.
\end{proof}

\begin{remark}
The enlargement induced by $\mathscr{K}_{\mathrm{hybrid}}$ is nontrivial.  In particular, the cone $\mathscr{K_{\mathrm{hybrid}}}$ includes functions whose Luxemburg norm is small relative to their sup norm; hence they satisfy the hybrid coercivity inequality but fail the sup--norm coercivity inequality. This ensures that the hybrid cone admits a broader class of functions than $\mathscr{K}_\infty$.
\end{remark}

Our next results establish two-sided localisations for $\|u\|_{L^{p(\cdot)}}$ provided that $u \in \partial \widehat V_{\rho}$. Such localisation results, in the context of the sup-norm, can be found in \cite{goodrich2025unified}. Here we obtain analogous bounds in the variable-exponent setting, formulated in terms of the Luxemburg norm while still retaining a connection to the supremum norm through the hybrid cone. In particular, the lower bound is derived via the hybrid thickness condition (i.e., Definition \ref{def:hybrid-thickness}) together with the sup-norm localisation, whereas the upper bound follows from a reverse H\"older--type argument combined with the coercivity relation $(\mathbf 1*u)(1)\geq C_0\|u\|_{L^{p(\cdot)}}$. These bounds provide the \textquotedblleft Luxemburg--norm annulus\textquotedblright\ within which fixed points will be localised, and they form a key foundation for the existence theory developed in the sequel.  We begin with a lower bound on $\Vert u\Vert_{L^{p(\cdot)}}$ for $u\in\partial\widehat{V}_{\rho}$.

\begin{lemma}\label{lem:Lp_lower_bound}
Assume condition \emph{(H2)}–\emph{(H3)} and let $\rho>0$. For any $\rho>0$, if $u\in \partial\widehat V_\rho$, then
\begin{equation*}
\|u\|_{L^{p(\cdot)}}\ge\eta_0\,(\beta-\alpha)^{\frac{1}{p^-}}\Vert u\Vert_\infty.
\end{equation*}
\end{lemma}

\begin{proof}
Let $\delta:=\|u\|_{L^{p(\cdot)}}$ and set $m:=\beta-\alpha$. Since $u\in\partial\widehat V_\rho\subset \mathscr K_{\mathrm{hybrid}}$, we have 
\begin{equation}
\min_{t\in[\alpha,\beta]}u(t)\ge\eta_0\Vert u\Vert_{\infty}\notag
\end{equation}
and, in part by Lemma \ref{lem:Luxemburg_basic},
\begin{equation}\label{eq2.3ggg}
1\ge \int_{0}^{1} \left(\frac{u(t)}{\delta}\right)^{p(t)}\,dt
 \ge \int_{\alpha}^{\beta} \biggl(\frac{u(t)}{\delta}\biggr)^{p(t)}\,dt
 \ge \int_{\alpha}^{\beta} \biggl(\frac{\eta_0 \|u\|_\infty}{\delta}\biggr)^{p(t)}\,dt .
\end{equation}

Suppose, for contradiction, that 
\[
\delta<\eta_0m^{\tfrac{1}{p^-}}\|u\|_\infty.
\]
Define 
\[
c:=\frac{\eta_0\|u\|_\infty}{\delta}.
\]
Note that we may assume without loss of generality that $\delta>0$, for otherwise $\Vert u\Vert_{L^{p(\cdot)}}=0$ and the desired conclusion is trivially true.  Then $c>m^{-1/p^-}\ge 1$, and hence (using $p(t)\ge p^->1$) we obtain
\[
\int_\alpha^\beta 
   \biggl(\frac{\eta_0 \|u\|_\infty}{\delta}\biggr)^{p(t)} \, dt
=
\int_\alpha^\beta c^{p(t)}\,dt\ge\int_\alpha^\beta c^{p^-}\,dt
= mc^{p^-}
>m\bigl(m^{-1/p^-}\bigr)^{p^-}
= 1,
\]
which is a contradiction to inequality \eqref{eq2.3ggg}. Therefore, $\delta\ge \eta_0m^{\frac{1}{p^-}}\|u\|_\infty$, as claimed.
\end{proof}

\begin{lemma}
\label{lem:Lp_lower_combined}
Assume conditions \emph{(H2)}–\emph{(H3)} and let $\rho>0$. If $u \in \partial \widehat V_\rho$, then
\begin{equation*}
\|u\|_{L^{p(\cdot)}}\ge\eta_0(\beta-\alpha)^{\frac{1}{p^-}}
\left[\left(\frac{\rho}{(b*\bm{1})(1)}\right)^{\!\frac{1}{p^+}} + \varepsilon_1(\rho,b)\right].
\end{equation*}
where $\varepsilon_1(\rho,b)$ is the piecewise function from Lemma~\ref{lem:sup-lower-bound}.
\end{lemma}

\begin{proof}
By Lemma~\ref{lem:Lp_lower_bound},
\begin{equation}
\|u\|_{L^{p(\cdot)}} \ge \eta_0 (\beta-\alpha)^{\frac{1}{p^-}} \|u\|_\infty.\notag
\end{equation}
By Lemma~\ref{lem:sup-lower-bound},
\begin{equation}
\|u\|_\infty \ge \left(\frac{\rho}{(b*\bm{1})(1)}\right)^{\frac{1}{p^+}}+\varepsilon_1(\rho,b).\notag
\end{equation}
Combine the two bounds to obtain the desired inequality.
\end{proof}

\begin{remark}
Lemma \ref{lem:Lp_lower_combined} is the Luxemburg–norm analogue of \cite[Lemma 2.3]{goodrich2025unified}.
\end{remark}

\begin{remark}
At first glance the estimate
\[
\|u\|_{L^{p(\cdot)}}\ge\eta_0\,(\beta-\alpha)^{\frac{1}{p^-}}\|u\|_\infty
\]
may look numerically \textquotedblleft weaker\textquotedblright\ because on finite measure sets one has $\|u\|_{L^{p(\cdot)}}\le\|u\|_\infty$.
This comparison is misleading for two structural reasons.
\begin{itemize}
\item \emph{Correct geometry:} The nonlocal coefficient depends on integrals of $u^{p(\cdot)}$. The Luxemburg norm is the natural gauge of the modular driving the equation; large $\|u\|_\infty$ on vanishing support can contribute almost nothing to $\displaystyle
\int_{0}^{1} |u(s)|^{p(s)}\,ds.$
\item \emph{Thickness:} The factor $(\beta-\alpha)^{\frac{1}{p^-}}$ ties the lower bound to a set of fixed measure on which $u$ is forced to be large. This excludes spike pathologies invisible to a pure $L^\infty$ analysis and, consequently, allows our framework to capture valid solutions without a large peak that would otherwise be missed.
\end{itemize}
Thus, Lemma~\ref{lem:Lp_lower_combined} provides the correct lower scale in the topology that governs the operator and the fixed-point argument.
\end{remark}

\begin{example}Fix $M>1$ and let $E\subseteq[0,1]$ be measurable. Define the test function
\[
u_M := M \bm{1}_E.
\]
Then
\[
I_{p(\cdot)}(u_M)
   = \int_E M^{p(x)}\,dx
   \in \bigl[M^{p^-}|E|,M^{p^+}|E| \,\bigr].
\]
Consequently,
\[
M\,|E|^{1/p^-} \le \|u_M\|_{L^{p(\cdot)}} \le M\,|E|^{1/p^+}.
\]
Hence, although $\|u_M\|_\infty=M$ for all non-null $E$, the Luxemburg norm depends on both the height $M$ and the measure of the support $|E|$. As $|E|\downarrow0$, $\|u_M\|_{L^{p(\cdot)}}\to0$, illustrating how the Luxemburg norm captures the effective ``spread'' of a function rather than its peak.
\end{example}

Fix a real number $q\in(1,p^-)$.  We now establish three auxiliary results: 
\begin{enumerate}
    \item a monotonicity property under exponent scaling (Lemma~\ref{lem:Lpq_le_Lp}) -- i.e.,
    \begin{equation}
    \|u\|_{L^{p(\cdot)/q}}\le\|u\|_{L^{p(\cdot)}};\notag
    \end{equation}
    \item a sharp lower bound based on the cone’s thickness condition (Lemma~\ref{lem:thick-to-Lpq}); and
    \item a reverse Hölder–type upper bound along $\partial\widehat V_\rho$ (Lemma~\ref{lem:Lux_upper_bound}). 
\end{enumerate}
Together these will play an important role in the proof of the existence theorem later in this section.

\begin{lemma}
\label{lem:Lpq_le_Lp}
Under the standing assumptions of this subsection (i.e., with $q\in(1,p^-)$ fixed),
for every $u\in L^{p(\cdot)}\big((0,1)\big)$ one has
\begin{equation*}\label{eq:Lpq-le-Lp}
\|u\|_{L^{\frac{p(\cdot)}{q}}} \le \|u\|_{L^{p(\cdot)}}.
\end{equation*}
\end{lemma}

\begin{proof}
Set $\displaystyle r(t):=\frac{p(t)}{q}$ and $C:=\|u\|_{L^{r(\cdot)}}$. Define $\displaystyle v(t):=\frac{|u(t)|}{C}$.
By the modular–norm identity for the Luxemburg norm (see Lemma~\ref{lem:Luxemburg_basic})
\[
\int_0^1 v(t)^{r(t)} \, dt =1. \qquad
\]
Because $q>1$, Jensen's inequality yields
\[
\int_0^1 v(t)^{p(t)} \, dt
  =\int_0^1 \bigl(v(t)^{r(t)}\bigr)^{q} \, dt
  \ge\Biggl(\int_0^1 v(t)^{r(t)} \, dt\Biggr)^{q}
  =1.
\]
That is,
\[
\int_0^1 \biggl(\frac{|u(t)|}{C}\biggr)^{p(t)} \, dt \ge 1.
\]

Let $f(\delta):=\displaystyle\int_0^1 \left(\frac{|u(t)|}{\delta}\right)^{p(t)} dt$ for $\delta>0$. Since $p(t)>1$, it follows that $f$ is strictly decreasing in $\delta$. By Lemma \ref{lem:Luxemburg_basic} we have that
\[
f\big(\|u\|_{L^{p(\cdot)}}\big)=1.
\]
Because $f(C)\ge 1=f\big(\|u\|_{L^{p(\cdot)}}\big)$ and $f$ is strictly decreasing, we conclude $C\le \|u\|_{L^{p(\cdot)}}$, i.e.
\[
\|u\|_{L^{\frac{p(\cdot)}{q}}}\le\|u\|_{L^{p(\cdot)}},
\]
as desired.
\end{proof}

\begin{remark}
The condition $q<p^-$ ensures $r(\cdot)=\frac{p(\cdot)}{q}>1$.  So, Lemmata \ref{lem:Luxemburg_basic} and \ref{lem:Lpq_le_Lp}, together with Corollary \ref{cor:mod_norm_ge_1}, apply without further hypotheses. In fact, neither continuity of $u$ nor of $p(\cdot)$ is technically required here.
\end{remark}

\begin{lemma}
\label{lem:thick-to-Lpq}
Assume conditions \textnormal{(H2)--(H3)} and fix $q\in(1,p^-)$. If $u\in \mathscr{K}_{\mathrm{hybrid}}$, then
\begin{equation}\label{eq:thick-to-Lpq}
\|u\|_{L^{\frac{p(\cdot)}{q}}}\ge \eta_0(\beta-\alpha)^{\frac{q}{p^-}}\|u\|_\infty,
\end{equation}
and, hence,
\begin{equation}\label{eq:thick-to-Lpq-vs-Lp}
\eta_0(\beta-\alpha)^{q/p^-}\|u\|_{L^{p(\cdot)}} \le \|u\|_{L^{\frac{p(\cdot)}{q}}}.
\end{equation}
\end{lemma}

\begin{proof}
Let $m:=\beta-\alpha\in(0,1]$ and set $\delta_*:=\eta_0\,m^{\,\frac{q}{p^-}}\|u\|_\infty$. On $[\alpha,\beta]$ it holds that $u(t)\ge \eta_0\|u\|_\infty$, and so,
\begin{align*}
\int_0^1 \biggl(\frac{u(t)}{\delta_*}\biggr)^{\tfrac{p(t)}{q}} \, dt
\ge
\int_\alpha^\beta m^{-\frac{p(t)}{p^-}} \, dt
\ge m \cdot m^{-1} \;=\; 1,
\end{align*}
seeing as $p(t)\ge p^-$. Writing
\begin{equation}
S(\delta):=\int_0^1 \left(\frac{u}{\delta}\right)^{\frac{p(t)}{q}}dt,\notag
\end{equation}
we have that $S$ strictly decreasing in $\delta$ and that
\begin{equation}
\|u\|_{L^{\frac{p(\cdot)}{q}}}=\inf\{\delta>0: S(\delta)\le 1\}.\notag
\end{equation}
Thus,
\begin{equation}
S(\delta_*)\ge 1\quad \text{implies that} \quad \delta_*\le \|u\|_{L^{\frac{p(\cdot)}{q}}},\notag
\end{equation}
proving inequality \eqref{eq:thick-to-Lpq}. Finally, since $\|u\|_{L^{p(\cdot)}}\le \|u\|_\infty$ on finite measure sets, inequality \eqref{eq:thick-to-Lpq-vs-Lp} follows at once.
\end{proof}

\begin{lemma}
\label{lem:Lux_upper_bound}
Assume conditions \textnormal{(H1)--(H3)}. Let $q\in(1,p^-)$ and suppose $b^{\frac{1}{1-q}}\in L^1((0,1])$.  For $\rho>0$ and $u\in\partial\widehat V_\rho$, set
\begin{equation*}
C_1(q):=\Biggl(\int_0^1\big(b(1-s)\big)^{\frac{1}{1-q}}\,ds\Biggr)^{1-q}
\end{equation*}
and
\begin{equation*}
K:=\frac{\rho}{C_1(q)}.
\end{equation*}
Then
\begin{equation*}
\|u\|_{L^{\frac{p(\cdot)}{q}}}
<
K^{\frac{1}{p^-}}+\varepsilon_2(K),
\qquad \text{where} \qquad
\varepsilon_2(K):=
\begin{cases}
0, & K\ge 1,\\
K^{\frac{1}{p^+}}-K^{\frac{1}{p^-}}, & 0<K<1.
\end{cases}
\end{equation*}
\end{lemma}

\begin{proof}
Since $u\in\partial\widehat V_\rho$,
\[
\rho = \int_0^1 b(1-s)u(s)^{p(s)}\,ds.
\]
Applying the reverse Hölder inequality gives
\[
\rho \ge
\Bigg(\int_0^1\big(b(1-s)\big)^{\frac{1}{1-q}}\,ds\Bigg)^{1-q}
\bigg(\int_0^1 u(s)^{\frac{p(s)}{q}}\,ds\bigg)^{q}
= C_1(q)\left(I_{\frac{p(\cdot)}{q}}(u)\right)^{q},
\]
where 
\[
I_{\frac{p(\cdot)}{q}}(u)
:= \int_0^1 |u(t)|^{\frac{p(t)}{q}}\,dt
\]
is the modular. Hence,
\begin{equation*}
I_{\frac{p(\cdot)}{q}}(u) \le K^{\frac{1}{q}}.
\end{equation*}
Equality in the reverse Hölder inequality would force $u^{\frac{p(\cdot)}{q}}$ to be proportional to $\big(b(1-\cdot)\big)^{\frac{1}{1-q}}$ a.e.; this does not occur for nontrivial $u$ on $\partial\widehat V_\rho$, and so, in fact
\begin{equation}\label{eq:Ik-bound}
I_{\frac{p(\cdot)}{q}}(u) < K^{\frac{1}{q}}.
\end{equation}

We now split by the size of $\displaystyle K:=\frac{\rho}{C_1(q)}$.

\smallskip

\emph{Case $0<K<1$.} From inequality \eqref{eq:Ik-bound}, the fact that $I_{p(\cdot)/q}(u)<1$, and Corollaries \ref{cor:mod_norm_ge_1} and \ref{cor:mod_norm_le_1}, we obtain that
\[
\|u\|_{L^{\frac{p(\cdot)}{q}}}<1 
\quad\text{(since } I_{p(\cdot)/q}(u)<1 \text{ iff } \|u\|_{L^{p(\cdot)/q}}<1\text{).}
\]
Consequently, by Corollary~\ref{cor:mod_norm_le_1},
\[
I_{\frac{p(\cdot)}{q}}(u)\ge \|u\|_{L^{\frac{p(\cdot)}{q}}}^{\frac{p^+}{q}}.
\]
Thus, $\|u\|_{L^{\frac{p(\cdot)}{q}}}<K^{\frac{1}{p^+}}$, i.e.,
\[
\|u\|_{L^{\frac{p(\cdot)}{q}}} 
< K^{\frac{1}{p^-}}+\big(K^{\frac{1}{p^+}}-K^{\frac{1}{p^-}}\big)
= K^{\frac{1}{p^-}}+\varepsilon_2(K).
\]

\emph{Case $K\ge1$.} Since $q>1$ and $K\ge1$, we have $K^{1/q}\ge1$. 
Together with inequality \eqref{eq:Ik-bound}, $I_{\frac{p(\cdot)}{q}}(u)$ may lie either below or above 1. For this reason, we distinguish between two subcases.

\smallskip
\noindent
$\bullet$\, If $\|u\|_{L^{\frac{p(\cdot)}{q}}}\ge1$, then by 
Corollary~\ref{cor:mod_norm_ge_1},
\[
I_{\frac{p(\cdot)}{q}}(u)
   \ge \|u\|_{L^{\frac{p(\cdot)}{q}}}^{\frac{p^-}{q}}.
\]
Combining this with inequality \eqref{eq:Ik-bound} yields
\[
\|u\|_{L^{\frac{p(\cdot)}{q}}}
   < K^{\frac{1}{p^-}}.
\]

\smallskip
\noindent
$\bullet$\, If $0<\|u\|_{L^{\frac{p(\cdot)}{q}}}<1$, then 
Corollary~\ref{cor:mod_norm_le_1} yields
\[
I_{\frac{p(\cdot)}{q}}(u)
   \ge \|u\|_{L^{\frac{p(\cdot)}{q}}}^{\frac{p^+}{q}},
\]
which, again by way of inequality \eqref{eq:Ik-bound}, implies that
\[
\|u\|_{L^{\frac{p(\cdot)}{q}}}
   < K^{\frac{1}{p^+}}
   \le K^{\frac{1}{p^-}}
   \quad (\text{since }K\ge1\text{ and } 1<p^-\le p^+).
\]

\smallskip
All in all, in either subcase we obtain the same estimate, namely
\[
\|u\|_{L^{\frac{p(\cdot)}{q}}}
   < K^{\frac{1}{p^-}}
   = K^{\frac{1}{p^-}}+\varepsilon_2(K),
\]
because $\varepsilon_2(K)=0$ for $K\ge1$. Hence, the unified bound holds in both regimes.
\end{proof}

Combining Lemma \ref{lem:Lux_upper_bound} with Lemma~\ref{lem:thick-to-Lpq} yields a uniform Luxemburg–norm bound on $u\in\widehat V_\rho$.  This is the content of Corollary~\ref{cor:Vrho_bounded}, which is essential for our fixed point approach since Lemma \ref{lem:fixed-point-index} requires that our sets be bounded.

\begin{corollary}\label{cor:Vrho_bounded}
For each $\rho>0$, the set $\widehat{V}_{\rho}$ is bounded in $\mathscr{C}([0,1])\cap L^{p(\cdot)}([0,1])$.
\end{corollary}

\begin{proof}
Fix $\rho>0$ and $u\in\widehat V_\rho$. By the definition of $\widehat V_\rho$, we have
\[
\left(b*u^{p(\cdot)}\right)(1) < \rho.
\]
In particular, if we set
\[
K':=\frac{\left(b*u^{p(\cdot)}\right)(1) }{C_1(q)} \in(0,+\infty),
\]
where
\begin{equation}
C_1(q):=\left(\int_0^1\big(b(1-s)\big)^{\frac{1}{1-q}}\ ds\right)^{1-q}\notag
\end{equation}
is as in the statement of Lemma \ref{lem:Lux_upper_bound}, then
\[
K' < \frac{\rho}{C_1(q)}.
\]
By Lemma~\ref{lem:Lux_upper_bound} (applied with $K'$), and since the bound there is
nondecreasing in $K$, we obtain the uniform estimate
\[
\|u\|_{L^{\frac{p(\cdot)}{q}}} 
   < M_{\frac{p}{q}}(\rho)
   := \left(\frac{\rho}{C_1(q)}\right)^{\!\frac{1}{p^-}}
      + \varepsilon_2\!\left(\frac{\rho}{C_1(q)}\right),
\]
where
\[
\varepsilon_2\!\left(\frac{\rho}{C_1(q)}\right)
   :=
   \begin{cases}
      0, & \left(\frac{\rho}{C_1(q)}\right)\ge 1,\\[3pt]
      \left(\frac{\rho}{C_1(q)}\right)^{\!\frac{1}{p^+}}
      -\left(\frac{\rho}{C_1(q)}\right)^{\!\frac{1}{p^-}},
      & 0<\frac{\rho}{C_1(q)}<1.
   \end{cases}
\]
By Lemma~\ref{lem:thick-to-Lpq}, this yields a uniform sup–norm bound:
\[
\|u\|_\infty \le \eta_0^{-1}(\beta-\alpha)^{-q/p^-}\,M_{\frac{p}{q}}(\rho)<+\infty.
\]
Finally, by Lemma~\ref{lem:Linf_to_Lp_embedding},
$\|u\|_{L^{p(\cdot)}}\le \|u\|_\infty$. Hence, $\widehat V_\rho$ is uniformly bounded in both
$\mathscr{C}([0,1])$ and $L^{p(\cdot)}([0,1])$.
\end{proof}

\begin{remark}
Corollary \ref{cor:Vrho_bounded} is the hybrid cone analogue of  \cite[Lemma 2.4]{goodrich2025unified}.
\end{remark}

\begin{remark}\label{remark2.22aaa}
For later estimates we record a consolidated pointwise bound on $u \in \widehat{V}_{\rho}.$

\smallskip
\noindent\emph{(i) Luxemburg-driven bound.}
By Lemma~\ref{lem:thick-to-Lpq} and Corollary~\ref{cor:Vrho_bounded} we have, for every 
$u \in \widehat{V}_{\rho}$,
\[
\|u\|_\infty \le B_{\infty,\rho}^{\mathrm{Lux}} 
\quad\text{with}\quad \displaystyle 
B_{\infty,\rho}^{\mathrm{Lux}}
:= \eta_0^{-1}(\beta-\alpha)^{-q/p^-}\,M_{\frac{p}{q}}(\rho),
\]
where, from Lemma~\ref{lem:Lux_upper_bound}, 
\[
M_{\frac{p}{q}}(\rho)=\left(\frac{\rho}{C_1(q)}\right)^{\frac{1}{p^-}}+\varepsilon_2\left(\frac{\rho}{C_1(q)}\right),
\qquad
C_1(q):=\left(\int_0^1\big(b(1-s)\big)^{\frac{1}{1-q}}\,ds\right)^{1-q}.
\]

\smallskip
\noindent\emph{(ii) Supremum–framework bound.}
From \cite[Lemma 2.4]{goodrich2025unified} one has for $u\in \widehat V_{\rho}$,
\[
\|u\|_\infty\le B_{\infty,\rho}^{\mathrm{Old}}
\quad\text{with}\quad
B_{\infty,\rho}^{\mathrm{Old}}:= C_0^{-1}2^{\frac{p^{+}-q}{p^-}}
\left[\rho^{\frac{1}{q}}\left(\left(b^{\frac{1}{1-q}}*\bm{1}\right)(1)\right)^{\frac{q-1}{q}}+1\right]^{\frac{q}{p^-}}.
\]

\smallskip
\noindent\emph{(iii) Consolidated bound.}
Using the definitions of $B_{\infty,\rho}^{\mathrm{Lux}}$ and 
$B_{\infty,\rho}^{\mathrm{Old}}$ above,
define
\[
B_{\infty,\rho}:=\min\left\{B_{\infty,\rho}^{\mathrm{Lux}},\; B_{\infty,\rho}^{\mathrm{Old}}\right\}.
\]
Then for every $u\in {\widehat V_{\rho}}$,
\[
\|u\|_\infty\le B_{\infty,\rho}.
\]
This $B_{\infty,\rho}$ depends only on the structural data $(\alpha,\beta,\eta_0,C_0,p^\pm,q)$,
on $b$ through both $C_1(q)$ and $(b^{\frac{1}{1-q}}*\bm{1})(1)$, and on the annulus parameter $\rho$,
but is independent of the particular $u$. 

It is instructive to compare the two estimates. For large $\rho$ the bounds, $B_{\infty,\rho}^{\mathrm{Lux}}$ and $B_{\infty,\rho}^{\mathrm{Old}}$, are of comparable size, and which one is sharper depends on the specific structural constants of the problem. In the small--$\rho$ regime ($\rho/C_1(q)<1$), the situation is clearer. The Luxemburg-driven estimate $B_{\infty,\rho}^{\mathrm{Lux}}$ satisfies $B_{\infty,\rho}^{\mathrm{Lux}}\to 0$ as $\rho\downarrow 0$, whereas $B_{\infty,\rho}^{\mathrm{Old}}$ tends to the constant $\displaystyle C_0^{-1}2^{\frac{p^+-q}{p^-}}>0$. Thus, for all sufficiently small $\rho$, the Luxemburg-based bound is guaranteed to be strictly tighter than the old sup--norm estimate.
\end{remark}

Because we will study \eqref{eq1.1} via a topological fixed point approach, we define the operator $T\ : \ \overline{\widehat{V}_{\rho_2}}\setminus\widehat{V}_{\rho_1}\rightarrow\mathscr{C}([0,1])$ by
\begin{equation*}
(Tu)(t):=\lambda\int_0^1\left(A\left(\left(b*u^{p(\cdot)}\right)(1)\right)\right)^{-1}G(t,s)f(s,u(s))\,ds.
\end{equation*}
Recalling from condition (H1.3) that $A(t)>0$ for $t\in\big[\rho_1,\rho_2\big]$, it follows that $T$ is well defined on its indicated domain.  Additionally, we note that a nontrivial fixed point of $T$
\begin{enumerate}
\item is a nontrivial solution of \eqref{eq1.1};

\item satisfies the boundary data generated by the Green's function $G$; and so

\item is a positive solution of the problem defined by parts (1)--(2).
\end{enumerate}
Our next lemma is a standard result regarding $T$.

\begin{lemma}
Assume conditions \emph{(H1)}--\emph{(H3)}. Then the operator $T$ is completely continuous on $\overline{\widehat{V}_{\rho_2}}\setminus \widehat{V}_{\rho_1}$. Moreover,
\[
T\!\left(\overline{\widehat{V}_{\rho_2}}\setminus \widehat{V}_{\rho_1}\right)
   \subseteq \mathscr{K}_{\mathrm{hybrid}}.
\]
\end{lemma}

\begin{proof}
Omitted; the argument is standard -- see, for example, \cite{goodrich_existence_2022,goodrich_topological_2021-1}, with the modification being that the convolution condition now involves 
$\|Tu\|_{L^{p(\cdot)}}$.  However, since $\|Tu\|_{L^{p(\cdot)}} \le \|Tu\|_\infty$, due to Lemma~\ref{lem:Linf_to_Lp_embedding}, the proof does not change in a materially significant manner.
\end{proof}

Finally, we state the topological fixed point theorem \cite[Lemma 2.1]{cianciaruso_solutions_2017-1} that we use in the proof of our existence result.

\begin{lemma}\label{lem:fixed-point-index}
Let $U$ be a bounded open set and, with $\mathscr{K}$ a cone in a real Banach space $\mathscr{X}$, suppose both that $U_{\mathscr{K}}:=U\cap\mathscr{K}\supseteq\{\bm{0}\}$ and that $\overline{U_{\mathscr{K}}}\neq\mathscr{K}$. Assume that $T:\overline{U_{\mathscr{K}}}\to\mathscr{K}$ is a compact map such that $x\neq Tx$ for each $x\in\partial U_{\mathscr{K}}$. Then the fixed point index $i_{\mathscr{K}}\!\left(T,U_{\mathscr{K}}\right)$ has the following properties.
\begin{enumerate}
\item If there exists $v\in\mathscr{K}\setminus\{\bm{0}\}$ such that $x\neq Tx+\lambda v$ for each $x\in\partial U_{\mathscr{K}}$ and each $\lambda>0$, then $i_{\mathscr{K}}\!\left(T,U_{\mathscr{K}}\right)=0$.

\item If $\mu x\neq Tx$ for each $x\in\partial U_{\mathscr{K}}$ and for each $\mu\ge1$, then $i_{\mathscr{K}}\!\left(T,U_{\mathscr{K}}\right)=1$.

\item If $i_{\mathscr{K}}\!\left(T,U_{\mathscr{K}}\right)\neq0$, then $T$ has a fixed point in $U_{\mathscr{K}}$.

\item Let $U^1$ be open in $X$ with $\overline{U_{\mathscr{K}}^1}\subseteq U_{\mathscr{K}}$. If $i_{\mathscr{K}}\!\left(T,U_{\mathscr{K}}\right)=1$ and $i_{\mathscr{K}}\!\left(T,U_{\mathscr{K}}^{1}\right)=0$, then $T$ has a fixed point in $U_{\mathscr{K}}\setminus\overline{U_{\mathscr{K}}^{1}}$. The same result holds if $i_{\mathscr{K}}\!\left(T,U_{\mathscr{K}}\right)=0$ and $i_{\mathscr{K}}\!\left(T,U_{\mathscr{K}}^{1}\right)=1$.
\end{enumerate}
\end{lemma}

\subsection{Existence Result and Discussion}

With the necessary preliminaries in place, we now prove an existence theorem for \eqref{eq1.1} subject to the boundary data encoded by the Green’s function $G$. For $\rho>0$, we introduce the following notation, recalling the notation $\displaystyle B_{\infty,\rho}:=\min\left\{B_{\infty,\rho}^{\mathrm{Lux}},\; B_{\infty,\rho}^{\mathrm{Old}}\right\}$ introduced in Remark \ref{remark2.22aaa}.
\begin{align*}
m_{\rho} 
&:= \left[\left(\frac{\rho}{(b*\bm{1})(1)}\right)^{\frac{1}{p^+}} 
              + \varepsilon_1(\rho,b)\right] \\[4pt]
M_{\rho} 
&:= \frac{1}{\eta_0\,(\beta-\alpha)^{\tfrac{q}{p^-}}}
      \left[\left(\frac{\rho}{\mathrm{C}_1(q)}\right)^{\tfrac{1}{p^-}} 
            + \varepsilon_2\left(\frac{\rho}{\mathrm{C}_1(q)}\right)\right]
            \\[4pt]
{C}_2(q) 
&:= \left(\int_0^1 \bigl|b(1-s)\bigr|^{\tfrac{q}{q-1}}\,ds\right)^{\tfrac{q-1}{q}}
   \\[8pt]
N_1 
&:= \frac{\lambda}{A(\rho_1)}
      f^{m}_{[\alpha,\beta]\times[\eta_0 m_{\rho_1},\,B_{\infty,\rho_1}]}
      \inf_{t\in[\alpha,\beta]}\ \int_{\alpha}^{\beta}\! G(t,s)\,ds
      \\[8pt]
Y_1 
&:= N_1(\beta-\alpha)^{\tfrac{q}{p^-}}
\end{align*}
These quantities will be used throughout the remainder of this paper.

\begin{theorem}\label{theorem2.25}
Suppose that each of conditions \textnormal{(H1)--(H3)} holds and that $b^{\frac{1}{1-q}}\in L^1((0,1])$ for some $q\in(1,p^-)$. In addition, suppose that the numbers $\rho_1$ and $\rho_2$ from condition $(\mathrm{H}1.3)$ are such that each of the following inequalities holds.
\vskip0.2cm
\begin{enumerate}
\item \begin{equation}\label{eq:theorem_condition1_explicit}
\begin{cases}
C_1(q) Y_1^{p^-}>\rho_1 & \text{if } \dfrac{\rho_1}{C_1(q)} \ge 1 \\
C_1(q) Y_1^{p^+}>\rho_1 & \text{if } 0 < \dfrac{\rho_1}{C_1(q)} < 1
\end{cases}
\end{equation}
\vskip0.2cm
\item 
\begin{equation}\label{eq:theorem_condition2_explicit}
\begin{cases}
C_2(q) \left(\left[\frac{\lambda f^{M}_{[0,1]\times[0,B_{\infty,\rho_2}]}}{A(\rho_2)}\right]\left\|\int_0^1 G(\cdot,s)\,ds\right\|_{L^{q\cdot p(\cdot)}}\right)^{p^+} < \rho_2 & \text{if } \dfrac{\rho_2}{C_2(q)} \ge 1 \\[6pt]
C_2(q) \left(\left[\frac{\lambda f^{M}_{[0,1]\times[0,B_{\infty,\rho_2}]}}{A(\rho_2)}\right]\left\|\int_0^1 G(\cdot,s)\,ds\right\|_{L^{q\cdot p(\cdot)}}\right)^{p^-} < \rho_2 & \text{if } 0 < \dfrac{\rho_2}{C_2(q)} < 1
\end{cases}
\end{equation}
\end{enumerate}
\vskip0.2cm
Then problem \eqref{eq1.1}, subject to the boundary data generated by the Green's function $G$, has at least one positive solution $u_0\in\widehat{V}_{\rho_2}\setminus\overline{\widehat{V}_{\rho_1}}$.  Moreover, $u_0$ satisfies the Luxemburg norm localisation
\begin{equation}
\eta_0(\beta-\alpha)^{\frac{1}{p^-}}m_{\rho_1}\le\Vert u\Vert_{L^{p(\cdot)}}\le M_{\rho_2}.
\end{equation}
\end{theorem}

\begin{proof}
We first observe that $\bm{0}\in\widehat V_{\rho_2}$, since $\displaystyle\left(b*\bm{0}^{p(\cdot)}\right)(1)=0$.  This ensures that the zero function lies in the admissible set as required by Lemma~\ref{lem:fixed-point-index}.  In addition, note by Corollary \ref{cor:Vrho_bounded} that each set $\widehat{V}_{\rho}$ is bounded, which is also necessary for the application of Lemma \ref{lem:fixed-point-index}.

Our first goal is to invoke part (1) of Lemma \ref{lem:fixed-point-index}.  Therefore, we must show that for every 
$u\in\partial\widehat V_{\rho_1}$ one has
\[
u \;\not\equiv\; Tu+\mu\bm{1}
\qquad\text{for all }\mu>0,
\]
where $\bm{1}\in\mathscr K_{\text{hybrid}}$ because $\eta_0,C_0\in(0,1]$.  So, for contradiction, suppose that there exists $u\in\partial\widehat V_{\rho_1}$, 
i.e.\
\[
\int_0^1 b(1-t)(u(t))^{p(t)}\,dt = \rho_1,
\]
and some $\mu>0$ such that, for each $t\in[0,1]$,
\[
u(t) = (Tu)(t) + \mu.
\]
Since $u\equiv Tu+\mu$ with $\mu>0$, we have $u(t)\ge (Tu)(t)$ for all $t\in[0,1]$. Because $p(t)>1$ and $b(t)\ge0$, it follows that
\[
\big(u(t)\big)^{p(t)}\ge\bigl((Tu)(t)\bigr)^{p(t)}
\]
so that
\begin{equation}\label{eq2.9aaa}
b(1-t)\big(u(t)\big)^{p(t)}\ge b(1-t)\bigl((Tu)(t)\bigr)^{p(t)}.
\end{equation}
Integrating inequality \eqref{eq2.9aaa} from $t=0$ to $t=1$ yields
\begin{equation}\label{eq2.10aaa}
\bigl(b*u^{p(\cdot)}\bigr)(1)\ge\left(b*(Tu)^{p(\cdot)}\right)(1).
\end{equation}
Since $u\in\partial\widehat V_{\rho_1}$, the left-hand side of inequality \eqref{eq2.10aaa} equals $\rho_1$, and so,
\begin{equation}\label{eq:rho1_ge_Tu_functional}
\rho_1 \ge \left(b*(Tu)^{p(\cdot)}\right)(1).
\end{equation}
By the {reverse H\"older} inequality,
\begin{equation}\label{eq2.13ggg}
\begin{split}
\left(b*(Tu)^{p(\cdot)}\right)(1) &\geq \left(\int_0^1 (b(1-t))^{\frac{1}{1-q}} dt\right)^{1-q} \left(\int_0^1 \big((Tu)(t)\big)^{\frac{p(t)}{q}}\,dt\right)^{\!q}
 \\
&= C_1(q)\left(I_{\frac{p(\cdot)}{q}}(Tu)\right)^q.
\end{split}
\end{equation}
\newcommand{\I}{[\alpha,\beta]}
Combining inequality \eqref{eq2.13ggg} with inequality \eqref{eq:rho1_ge_Tu_functional} we see that
\begin{equation}\label{eq:rho1_ge_Cb_I_r}
\rho_1 \ge C_1(q) \left(I_{\frac{p(\cdot)}{q}}(Tu)\right)^q.
\end{equation}

We now show that inequality \eqref{eq:rho1_ge_Cb_I_r}, together with the hypotheses in the statement of the theorem, leads to a contradiction.  So, to begin the process of deriving the contradiction, recall the notation from Remark \ref{remark2.22aaa} and observe that
\begin{equation}\label{Tu-min-integral}
\begin{split}
(Tu)(t)
&\ge
\frac{\lambda}{A(\rho_1)}\int_\alpha^\beta G(t,s)f(s,u(s))\,ds\\ 
&\ge
\frac{\lambda}{A(\rho_1)}\int_\alpha^\beta G(t,s)f^m_{[\alpha,\beta]\times[\eta_0 m_{\rho_1},\,B_{\infty,\rho_1}]}\,ds\\
&\ge
\frac{\lambda}{A(\rho_1)}f^m_{[\alpha,\beta]\times[\eta_0 m_{\rho_1},\,B_{\infty,\rho_1}]}\inf_{t\in[\alpha,\beta]}\ \int_{\alpha}^{\beta} G(t,s)\,ds\\
&=N_1.
\end{split}
\end{equation}
Note that in inequality \eqref{Tu-min-integral} we have used the fact that, for each $s\in[\alpha,\beta]$,
\begin{equation}
\eta_0m_{\rho_1}=\eta_0\left[\left(\frac{\rho_1}{(b*\bm{1})(1)}\right)^{\frac{1}{p^+}}+\varepsilon_1\left(\rho_1,b\right)\right]\le\eta_0\Vert u\Vert_{\infty}\le u(s)\le\Vert u\Vert_{\infty}\le B_{\infty,\rho_1},\notag
\end{equation}
which follows from Lemma \ref{lem:sup-lower-bound}, Remark \ref{remark2.22aaa}, and the fact that $u\in\mathscr{K}_{\text{hybrid}}$.

We next use inequality \eqref{Tu-min-integral} to establish a lower bound for $\Vert Tu\Vert_{L^{\frac{p(\cdot)}{q}}}$.  Put $\displaystyle r(t):=\frac{p(t)}{q}$. Let $k$ range over $(0,+\infty)$ and consider the set
\[
\left\{\,k>0\ :\ \int_0^1 \left(\frac{|(Tu)(t)|}{k}\right)^{r(t)}\,dt \le 1\,\right\}.
\]
Assume, for sake of contradiction, that there exists some $k_0>0$, with $k_0<N_1(\beta-\alpha)^{1/r^-}$, that belongs to this set. If such a $k_0$ exists, then by the definition of the Luxemburg norm, we would have $\Vert Tu\Vert_{L^{r(\cdot)}} \le k_0$. However, if $k_0 < N_1 (\beta - \alpha)^{1/r^-}$, then it immediately follows that
\[ \int_{\alpha}^{\beta} \left(\frac{N_1}{k_0}\right)^{r(t)}\,dt > 1, \]
which since by inequality \eqref{Tu-min-integral} we know that $(Tu)(t) \ge N_1$, implies that
\[ \int_0^1 \left(\frac{|(Tu)(t)|}{k_0}\right)^{r(t)}\,dt\ge \int_{\alpha}^{\beta} \left(\frac{(Tu)(t)}{k_0}\right)^{r(t)}\,dt \ge \int_{\alpha}^{\beta} \left(\frac{N_1}{k_0}\right)^{r(t)}\,dt > 1. \]
This leads to the conclusion that
$$\int_0^1 \left(\frac{|(Tu)(t)|}{k_0}\right)^{r(t)}\,dt> 1.$$
But this contradicts our initial assumption that $k_0$ belongs to the set $$\left\{ k > 0 : \int_0^1 \left(\frac{|(Tu)(t)|}{k}\right)^{r(t)}\,dt \le 1 \right\}.$$
Therefore, no such $k_0$ can exist. All $k$ values that satisfy the condition
$$\int_0^1 \left(\frac{|(Tu)(t)|}{k}\right)^{r(t)}\,dt \le 1$$
must be greater than or equal to
\begin{equation}
N_1 (\beta - \alpha)^{1/r^-}.\notag
\end{equation}
By the definition of the infimum, this means
\begin{equation}\label{eq:lower_norm_bound}
\|Tu\|_{L^{r(\cdot)}} = \|Tu\|_{L^{\frac{p(\cdot)}{q}}}\ge N_1(\beta-\alpha)^{1/r^-}= N_1(\beta-\alpha)^{q/p^-}= Y_1.
\end{equation}

We now distinguish two complementary regimes for the ratio $\rho_1 / C_1(q)$, since either Corollary \ref{cor:mod_norm_ge_1} or Corollary \ref{cor:mod_norm_le_1} applies depending on whether $\|Tu\|_{L^{\frac{p(\cdot)}{q}}}$ lies above or below $1$. Specifically, recall from \eqref{eq:theorem_condition1_explicit} and \eqref{eq:rho1_ge_Cb_I_r} that
\[
C_1(q)\left(I_{\frac{p(\cdot)}{q}}(Tu)\right)^q \le \rho_1 < 
C_1(q) Y_1^{p^{\pm}},
\]
and so, the comparison between $\rho_1$ and $C_1(q)$ determines whether the modular is evaluated in the “large-norm” or “small-norm” regime, and we show that in each case we are led to our desired contradiction.

Accordingly, we treat the two possibilities separately:
\subsubsection*{Case 1: $\rho_1/C_1(q)\ge 1$}
Since the hypothesis in the statement of the theorem gives $C_1(q)Y_1^{p^-}>\rho_1$, we have $Y_1>1$. From \eqref{eq:lower_norm_bound} we know that $\|Tu\|_{L^{\frac{p(\cdot)}{q}}}\ge Y_1$.  Hence, $\|Tu\|_{L^{\frac{p(\cdot)}{q}}}>1$. By Corollary \ref{cor:mod_norm_ge_1}, we then have that
\[
I_{\frac{p(\cdot)}{q}}(Tu)
  \ge
  \|Tu\|_{L^{\frac{p(\cdot)}{q}}}^{\frac{p^-}{q}}
  \ge
  Y_1^{\frac{p^-}{q}}.
\]
From \eqref{eq:theorem_condition1_explicit} and \eqref{eq:rho1_ge_Cb_I_r} we also have that
\begin{equation}\label{eq:rho1-upper-contradiction}
C_1(q) Y_1^{p^-} > \rho_1 \ge C_1(q)\left(I_{\frac{p(\cdot)}{q}}(Tu)\right)^q
\;\implies\;
Y_1^{p^-} > \left(I_{\frac{p(\cdot)}{q}}(Tu)\right)^q.
\end{equation}
But \eqref{eq:rho1-upper-contradiction} yields
\begin{equation*}
Y_1^{p^-}>\left(I_{\frac{p(\cdot)}{q}}(Tu)\right)^q \ge Y_1^{\,q \cdot {\frac{p^-}{q}}} = Y_1^{p^-},
\end{equation*}
which is a contradiction. Hence, no $u\in\partial\widehat V_{\rho_1}$ satisfies $u\equiv Tu+\mu\bm{1}$ with $\mu>0$ in this case.

\subsubsection*{Case 2: $0 < \rho_1/C_1(q) < 1$}
From both \eqref{eq:theorem_condition1_explicit} and \eqref{eq:rho1_ge_Cb_I_r}, we see that
\[
C_1(q) Y_1^{p^+} > \rho_1 \ge C_1(q)\left(I_{\frac{p(\cdot)}{q}}(Tu)\right)^q
\;\implies\;
Y_1^{p^+} > \left(I_{\frac{p(\cdot)}{q}}(Tu)\right)^q.
\]
The case assumption $0 < \rho_1/C_1(q) < 1$ implies $I_{\frac{p(\cdot)}{q}}(Tu) < 1$; 
hence, $\|Tu\|_{L^{\frac{p(\cdot)}{q}}} < 1$. 
For this regime, Corollary \ref{cor:mod_norm_le_1} gives
\[
I_{\frac{p(\cdot)}{q}}(Tu) \ge \|Tu\|_{L^{\frac{p(\cdot)}{q}}}^{\frac{p^+}{q}}.
\]
Combining with the lower bound $\|Tu\|_{L^{\frac{p(\cdot)}{q}}} \ge Y_1$ from 
\eqref{eq:lower_norm_bound} yields
\[
I_{\frac{p(\cdot)}{q}}(Tu) \ge Y_1^{\frac{p^+}{q}}.
\]
Therefore,
\[
Y_1^{p^+}>\left(I_{\frac{p(\cdot)}{q}}(Tu)\right)^q \ge Y_1^{\,q\cdot {\frac{p^+}{q}}} = Y_1^{p^+},
\]
which yields the desired contradiction. Thus, no $u\in\partial\widehat V_{\rho_1}$ satisfies $u\equiv Tu+\mu\bm{1}$ with $\mu>0$ in this case. Since both cases lead to a logical contradiction, our initial supposition that $u = Tu + \mu\mathbf{1}$ must be false when $u \in \partial\widehat{V}_{\rho_1}$.

Therefore, we have successfully shown that whenever $u \in \partial\widehat{V}_{\rho_1}$, it follows that $u \not\equiv Tu + \mu\mathbf{1}$ for all $\mu > 0$. Hence, from part (1) of Lemma \ref{lem:fixed-point-index} we conclude that
\begin{equation}\label{eq:index-zero}
i_{\mathscr{K}}\!\left(T,\widehat{V}_{\rho_1}\right)=0.
\end{equation}

In order to invoke part (2) of Lemma \ref{lem:fixed-point-index}, we next demonstrate that whenever $u\in\partial\widehat{V}_{\rho_2}$ and $\mu\ge1$ it follows that $Tu\not\equiv\mu u$.  
For contradiction suppose that there exist $u\in\partial\widehat{V}_{\rho_2}$ and $\mu\ge1$ such that $(Tu)(t)=\mu u(t)$ for each $t\in[0,1]$.  Then for each $0\le t\le 1$ it follows that
\begin{equation}\label{eq2.19mmm}
\big(\mu u(t)\big)^{p(t)}=\big((Tu)(t)\big)^{p(t)}.
\end{equation}
Since $u\in\partial\widehat V_{\rho_2}$, we have $\displaystyle A\left(\left(b*u^{p(\cdot)}\right)(1)\right)=A(\rho_2)$. Define
\[
g(s):=\frac{f(s,u(s))}{\,f^{M}_{[0,1]\times[0,B_{\infty,\rho_2}]}\,}\in[0,1],
\]
where $B_{\infty,\rho_2}$ is the uniform cap from Remark \ref{remark2.22aaa}. Then $f\big(s,u(s)\big) = g(s)f^{M}_{[0,1]\times[0,B_{\infty,\rho_2}]}$, and so,
\[
(Tu)(t)
= \frac{\lambda}{A(\rho_2)}\int_0^1 G(t,s)f(s,u(s))\,ds
= \left[\frac{\lambda f^{M}_{[0,1]\times[0,B_{\infty,\rho_2}]}}{A(\rho_2)}\right]
   \int_0^1 G(t,s)g(s)\,ds .
\]
Since $0\le g(s)\le 1$ and $G(t,s)\ge0$, we obtain the following bound:
\[
0 \le \frac{(Tu)(t)}{\left[\frac{\lambda\,f^{M}_{[0,1]\times[0,B_{\infty,\rho_2}]}}{A(\rho_2)}\right]}
   = \int_0^1 G(t,s)g(s)\,ds
   \le \int_0^1 G(t,s)\,ds,
\]
where we use that $G(t,s)\ge0$ and $0\le g(s)\le 1$. By lattice monotonicity (that is, if $\big|v(t)\big|\le\big|w(t)\big|$ for $v,w\in C[0,1]\cap L^{p(\cdot)}$, then $\|v\|_{L^{p(\cdot)}}\le\|w\|_{L^{p(\cdot)}}$) and the homogeneity of the Luxemburg norm, we obtain
\begin{equation}\label{eq:Tu-Lqp-upper}
\|Tu\|_{L^{q\cdot p(\cdot)}}
\le
\left[\frac{\lambda f^{M}_{[0,1]\times[0,B_{\infty,\rho_2}]}}{A(\rho_2)}\right]\,
\left\|\int_0^1 G(\cdot,s)\,ds\right\|_{L^{q\cdot p(\cdot)}}.
\end{equation}
Using the fact that $\displaystyle\left(b*u^{p(\cdot)}\right)(1)=\rho_2$, the monotonicity of $x\mapsto x^{p(t)}$ for fixed $p(t)>1$, Hölder’s inequality, and equality \eqref{eq2.19mmm}, we obtain
\begin{equation}\label{eq:rho2-upper-bound}
\begin{split}
\rho_2&=\int_0^1b(1-t)\big(u(t)\big)^{p(t)}\ dt\\
&\le\int_0^1b(1-t)\big(\mu u(t)\big)^{p(t)}\ dt\\
&=\int_0^1b(1-t)\big(Tu(t)\big)^{p(t)}\ dt\\
&\le\left(\int_0^1 \big|b(1-t)\big|^{\frac{q}{q-1}}\,dt\right)^{\!\frac{q-1}{q}} \left( \int_0^1 \big|(Tu)(t)\big|^{q\cdot p(t)}\,dt \right)^{1/q}\\
&=C_2(q)\left(I_{q\cdot p(\cdot)}(Tu)\right)^{\frac{1}{q}}.
\end{split}
\end{equation}

To complete the proof, it remains to show that inequalities~\eqref{eq:Tu-Lqp-upper} and~\eqref{eq:rho2-upper-bound} jointly lead to a contradiction.  To proceed, we appeal to Corollaries~\ref{cor:mod_norm_ge_1} and~\ref{cor:mod_norm_le_1}, which describe how the modular and the Luxemburg norm compare depending on whether $\|Tu\|_{L^{q\cdot p(\cdot)}}$ is greater than or less than $1$. Accordingly, we consider two separate cases based on the size of the quantity \(\dfrac{\rho_2}{C_2(q)}\).

\subsubsection*{Case 1: \(\dfrac{\rho_2}{C_2(q)} \ge 1\)}
From \eqref{eq:theorem_condition2_explicit} and \eqref{eq:rho2-upper-bound}, we see that
\begin{equation}\label{eq:case1-second}
\begin{split}
C_2(q)\Bigg(
   \frac{\lambda f^{M}_{[0,1]\times[0,B_{\infty,\rho_2}]}}{A(\rho_2)}\,
   \left\|\int_0^1 G(\cdot,s)\,ds\right\|_{L^{q\cdot p(\cdot)}}
\Bigg)^{p^+}
   &< \rho_2
   \le C_2(q)\left(I_{q\cdot p(\cdot)}(Tu)\right)^{\frac{1}{q}}
   \\
\implies\quad
\Bigg(
   \frac{\lambda f^{M}_{[0,1]\times[0,B_{\infty,\rho_2}]}}{A(\rho_2)}\,
   \left\|\int_0^1 G(\cdot,s)\,ds\right\|_{L^{q\cdot p(\cdot)}}
\Bigg)^{p^+}
   &< \left(I_{q\cdot p(\cdot)}(Tu)\right)^{\frac{1}{q}}.
\end{split}
\end{equation}
Also, from inequality \eqref{eq:case1-second} we see that $\dfrac{\rho_2}{C_2(q)}\ge 1$ implies that $\big(I_{q\cdot p(\cdot)}(Tu)\big)^{\frac{1}{q}}\ge 1$. Now, if $\|Tu\|_{L^{q\cdot p(\cdot)}}<1$, then Corollary \ref{cor:mod_norm_le_1} yields
\begin{align*}
\big(I_{q\cdot p(\cdot)}(Tu)\big)^{\frac{1}{q}}
   \le \|Tu\|_{L^{q\cdot p(\cdot)}}^{p^-}
   < 1,
\end{align*}
\noindent which is a contradiction. Hence, it can only be that $\|Tu\|_{L^{q\cdot p(\cdot)}}\ge 1$, and so, by Corollary~\ref{cor:mod_norm_ge_1},
\begin{equation}\label{eq2.22ggg}
\big(I_{q\cdot p(\cdot)}(Tu)\big)^{\frac{1}{q}}\le \|Tu\|_{L^{q\cdot p(\cdot)}}^{p^+}.
\end{equation}
But then combining inequality \eqref{eq2.22ggg} with inequality~\eqref{eq:Tu-Lqp-upper} we obtain
\[
\left(I_{q\cdot p(\cdot)}(Tu)\right)^{\frac{1}{q}}
   \le 
   \Bigg(
       \frac{\lambda f^{M}_{[0,1]\times[0,B_{\infty,\rho_2}]}}{A(\rho_2)}\,
       \left\|\int_0^1 G(\cdot,s)\,ds\right\|_{L^{q\cdot p(\cdot)}}
   \Bigg)^{p^+},
\]
which contradicts \eqref{eq:case1-second}. Therefore, no $u\in\partial\widehat V_{\rho_2}$ satisfies $\mu u\equiv Tu$ with $\mu\ge 1$ in this case.

\subsubsection*{Case 2: $0<\dfrac{\rho_2}{C_2(q)} < 1$}
From \eqref{eq:theorem_condition2_explicit} and \eqref{eq:rho2-upper-bound}, we have
\begin{equation}\label{eq:case2-first}
\begin{split}
C_2(q)\Bigg(
   \frac{\lambda f^{M}_{[0,1]\times[0,B_{\infty,\rho_2}]}}{A(\rho_2)}\,
   \left\|\int_0^1 G(\cdot,s)\,ds\right\|_{L^{q\cdot p(\cdot)}}
\Bigg)^{p^-}
   &< \rho_2
   \le C_2(q)\,\left(I_{q\cdot p(\cdot)}(Tu)\right)^{\frac{1}{q}}\\
\implies\quad
\Bigg(
   \frac{\lambda f^{M}_{[0,1]\times[0,B_{\infty,\rho_2}]}}{A(\rho_2)}\,
   \left\|\int_0^1 G(\cdot,s)\,ds\right\|_{L^{q\cdot p(\cdot)}}
\Bigg)^{p^-}
   &< \left(I_{q\cdot p(\cdot)}(Tu)\right)^{\frac{1}{q}}.
\end{split}
\end{equation}
Because $\dfrac{\rho_2}{C_2(q)}<1$, condition \eqref{eq:theorem_condition2_explicit} yields
\begin{equation}\label{eq:rho2-upper-smallK}
\Bigg(
   \frac{\lambda f^{M}_{[0,1]\times[0,B_{\infty,\rho_2}]}}{A(\rho_2)}\,
   \left\|\int_0^1 G(\cdot,s)\,ds\right\|_{L^{q\cdot p(\cdot)}}
\Bigg)^{p^-}
< \frac{\rho_2}{C_2(q)} < 1.
\end{equation}
Since $p^->1$, the map $t\mapsto t^{p^-}$ is strictly increasing on $[0,+\infty)$, and so, by \eqref{eq:Tu-Lqp-upper} it follows that
\[
\|Tu\|_{L^{q\cdot p(\cdot)}} 
\le
\frac{\lambda f^{M}_{[0,1]\times[0,B_{\infty,\rho_2}]}}{A(\rho_2)}\,
\left\|\int_0^1 G(\cdot,s)\,ds\right\|_{L^{q\cdot p(\cdot)}}
<1.
\]
Hence, by Corollary~\ref{cor:mod_norm_le_1},
\begin{equation}\label{eq2.26ggg}
\big(I_{q\cdot p(\cdot)}(Tu)\big)^{\frac{1}{q}}
\le \|Tu\|_{L^{q\cdot p(\cdot)}}^{p^-}.
\end{equation}
Combining inequality \eqref{eq2.26ggg} with \eqref{eq:Tu-Lqp-upper}, \eqref{eq:rho2-upper-bound}, \eqref{eq:case2-first}, and \eqref{eq:rho2-upper-smallK} yields
\begin{align*}
\big(I_{q\cdot p(\cdot)}(Tu)\big)^{\frac{1}{q}}
&\le \|Tu\|_{L^{q\cdot p(\cdot)}}^{\,p^-} \\
&\le 
\left(
   \frac{\lambda f^{M}_{[0,1]\times[0,B_{\infty,\rho_2}]}}{A(\rho_2)}
   \left\|\int_0^1 G(\cdot,s)\,ds\right\|_{L^{q\cdot p(\cdot)}}
\right)^{p^-} \\
&< \frac{\rho_2}{C_2(q)} \\
&\le \left(I_{q\cdot p(\cdot)}(Tu)\right)^{\frac{1}{q}},
\end{align*}
which is a contradiction. Therefore, no $u\in\partial\widehat V_{\rho_2}$ satisfies $Tu\equiv\mu u$ with $\mu\ge1$ in this case as well, and so, we can conclude from part (2) of Lemma \ref{lem:fixed-point-index} that
\begin{equation}\label{eq:index-one}
i_{\mathscr{K}}\!\left(T,\widehat{V}_{\rho_2}\right)=1.
\end{equation}

All in all, combining both \eqref{eq:index-zero} and \eqref{eq:index-one}, we infer from part (4) of Lemma \ref{lem:fixed-point-index} that there exists
\begin{equation*}
u_0\in \widehat{V}_{\rho_2}\setminus\overline{\widehat{V}_{\rho_1}}
\end{equation*}
such that $Tu_0\equiv u_0$. Hence, $u_0$ is a positive solution of \eqref{eq1.1}, subject to the boundary data encoded by the Green’s function $G$. Moreover, a combined application of Lemmata \ref{lem:Lp_lower_combined}, \ref{lem:thick-to-Lpq}, and \ref{lem:Lux_upper_bound} yields the Luxemburg norm localisation
\begin{align}\label{Lp(.) bound}
\eta_0(\beta-\alpha)^{\tfrac{1}{p^-}}
\left[
\left(\frac{\rho_1}{(b*\bm{1})(1)}\right)^{\frac{1}{p^+}}
+ \varepsilon_1(\rho_1,b)
\right]
&\le\Vert u\Vert_{L^{p(\cdot)}} \notag\\
&\le
\frac{1}{\eta_0(\beta-\alpha)^{\tfrac{q}{p^-}}}
\left[
\left(\frac{\rho_2}{C_1(q)}\right)^{\frac{1}{p^-}}
+ \varepsilon_2\left(\frac{\rho_2}{C_1(q)}\right)
\right],
\end{align}
as claimed. And this completes the proof of the theorem.
\end{proof}

\begin{remark}
In addition to the refined constant $N_1$ based on the integrated estimate 
$$\inf_{t\in[\alpha,\beta]} \int_{\alpha}^{\beta} G(t,s)\,ds,$$ 
one may also define the more elementary quantity
\[
N_0
:= \frac{\lambda}{A(\rho_1)}\eta_0
f^{m}_{[\alpha,\beta]\times[\eta_0 m_{\rho_1},\,B_{\infty,\rho_1}]}\,
\int_{\alpha}^{\beta}\mathscr{G}(s)\,ds,
\]
which takes advantage of assumption (H3.1). Since 
\begin{equation}
\inf_{t\in[\alpha,\beta]} \int_{\alpha}^{\beta} G(t,s)\,ds \ge \eta_0\int_{\alpha}^{\beta}\mathscr{G}(s)\,ds,\notag
\end{equation}
it follows that $N_1\ge N_0$. Hence, our use of $N_1$ represents a refinement, adopted here to obtain the strongest possible localisation result. The practical advantage of this refinement will be illustrated in Example~\ref{example2.12}.
\end{remark}

In the general setting of \textnormal{(H1)--(H3)} with no boundedness assumption on $b$, inequality \eqref{Lp(.) bound} provides the appropriate $q$–dependent upper localisation. However, when $b$ is uniformly bounded above and below, i.e.\ $0<b_*\le b(t)\le b^*<+\infty$, the $q$–dependence can be removed. In this regime we obtain the following stronger, $q$–free localisation result, which combines both sides of the annulus into a single statement.

\begin{lemma}
\label{lem:qfree_uniform_b_bounds}
Assume conditions \textnormal{(H1)--(H3)} and, in particular, that $0<b_*\le b(t)\le b^*<+\infty$ on $(0,1]$ for some real constants $b_*$ and $b^*$. 
For $\rho>0$ and $u\in\partial\widehat V_\rho$ one has
\begin{equation}\label{localisation for bounded b}
\left(\frac{\rho}{b^*}\right)^{\frac{1}{p^+}}+\varepsilon_1\left(\frac{\rho}{b^*}\right)\le\|u\|_{L^{p(\cdot)}}\le\left(\frac{\rho}{b_*}\right)^{\frac{1}{p^-}}+\varepsilon_2\left(\frac{\rho}{b_*}\right),
\end{equation}
where, for $\tau>0$,
\[
\varepsilon_1(\tau):=\begin{cases}
0, & \tau\ge 1\\[2pt]
\tau^{\frac{1}{p^-}}-\tau^{\frac{1}{p^+}}, & 0<\tau<1
\end{cases}
\qquad\text{ and }\qquad
\varepsilon_2(\tau):=\begin{cases}
0, & \tau\ge 1\\[2pt]
\tau^{\frac{1}{p^+}}-\tau^{\frac{1}{p^-}}, & 0<\tau<1
\end{cases}
.
\]
\end{lemma}

\begin{proof}
On $\partial\widehat V_\rho$ we have that
\[
\rho=\int_0^1 b(1\!-\!s)\,u(s)^{p(s)}\,ds.
\]
Using $0<b_*\le b(1-s)\le b^*$, it follows that
\begin{equation}\label{modular sandwich with bounded b}
\frac{\rho}{\,b^*}\le\int_0^1 u(s)^{p(s)}\,ds=I_{p(\cdot)}(u)\le\frac{\rho}{\,b_*}.
\end{equation}

We first establish the lower bound claimed in the statement of the lemma. Set $\tau_-:=\rho/b^*$. We split by the size of $\tau_-$.

\emph{Case $0<\tau_-<1$.} Consider the two possibilities for $\|u\|_{L^{p(\cdot)}}$:
\[
\begin{aligned}
&\text{if }\ \|u\|_{L^{p(\cdot)}}<1:\quad I_{p(\cdot)}(u)\le \|u\|_{L^{p(\cdot)}}^{p^-}
\ \Longrightarrow\ 
\|u\|_{L^{p(\cdot)}} \ge\left(I_{p(\cdot)}(u)\right)^{1/p^-}\ge\tau_-^{1/p^-},\\[2pt]
&\text{if }\ \|u\|_{L^{p(\cdot)}}\ge1:\quad 
\|u\|_{L^{p(\cdot)}} \ge \tau_-^{1/p^-} \, \text{trivially}.
\end{aligned}
\]
Hence, in either regime it holds that
\[
\|u\|_{L^{p(\cdot)}}\ge\tau_-^{1/p^-}=\tau_-^{1/p^+}+\left(\tau_-^{1/p^-}-\tau_-^{1/p^+}\right)=\tau_-^{1/p^+}+\varepsilon_1(\tau_-).
\]

\emph{Case $\tau_-\ge1$.} From inequality \eqref{modular sandwich with bounded b} it follows that $I_{p(\cdot)}(u)\ge \tau_-\ge 1$.  
First, since in this case we have that $I_{p(\cdot)}(u)\ge 1$, it follows that that $\|u\|_{L^{p(\cdot)}}\ge 1$ (see \cite{diening_lebesgue_2011}); and for $\|u\|_{L^{p(\cdot)}}\ge 1,$ Corollary \ref{cor:mod_norm_ge_1} gives
\begin{equation}\label{eq2.30ggg}
I_{p(\cdot)}(u) \le \|u\|_{L^{p(\cdot)}}^{p^+}.
\end{equation}
Combining inequality \eqref{eq2.30ggg} with $\tau_-\le I_{p(\cdot)}(u)$ yields
\[
\tau_- \le I_{p(\cdot)}(u) \le \|u\|_{L^{p(\cdot)}}^{p^+}
\ \Longrightarrow\ 
\|u\|_{L^{p(\cdot)}}\ge\tau_-^{1/p^+}
=\tau_-^{1/p^+}+\varepsilon_1(\tau_-),
\]
since $\varepsilon_1(\tau_-)=0$ for $\tau_-\ge1$.  And so we have the desired lower bound.

We next establish the upper bound claimed in the statement of the lemma. Set $\tau_+:=\rho/b_*$. Similar to the proof of the lower bound, we split by the size of $\tau_+$.

\emph{Case $0<\tau_+<1$.} By inequality \eqref{modular sandwich with bounded b} note that $I_{p(\cdot)}(u)\le\tau_+<1$, which forces $\|u\|_{L^{p(\cdot)}}<1$, for otherwise
$I_{p(\cdot)}(u)\ge\|u\|_{L^{p(\cdot)}}^{p^-}\ge1$. Hence, with $\|u\|_{L^{p(\cdot)}}<1$,
\[
\|u\|_{L^{p(\cdot)}}^{p^+}\le I_{p(\cdot)}(u)\le \tau_+
\ \Longrightarrow\
\|u\|_{L^{p(\cdot)}}\le \tau_+^{1/p^+}
= \tau_+^{1/p^-}+\big(\tau_+^{1/p^+}-\tau_+^{1/p^-}\big)
= \tau_+^{1/p^-}+\varepsilon_2(\tau_+).
\]

\emph{Case $\tau_+\ge1$.} We consider the two possibilities for $\|u\|_{L^{p(\cdot)}}$:
\[
\begin{aligned}
&\text{if }\ \|u\|_{L^{p(\cdot)}}<1:\quad \|u\|_{L^{p(\cdot)}}^{p^+}\le I_{p(\cdot)}(u)\le\tau_+
\ \Longrightarrow\ 
\|u\|_{L^{p(\cdot)}}\le\tau_+^{1/p^+}\le \tau_+^{1/p^-},\\[2pt]
&\text{if }\ \|u\|_{L^{p(\cdot)}}\ge1:\quad \|u\|_{L^{p(\cdot)}}^{p^-}\le I_{p(\cdot)}(u)\le \tau_+
\ \Longrightarrow\ 
\|u\|_{L^{p(\cdot)}} \le \tau_+^{1/p^-}.
\end{aligned}
\]
Thus, in either subcase we deduce that
\[
\|u\|_{L^{p(\cdot)}}\le \tau_+^{1/p^-}
= \tau_+^{1/p^-}+\varepsilon_2(\tau_+),
\]
since $\varepsilon_2(\tau_+)=0$ for $\tau_+\ge1$.

Combining the four cases yields the claimed bound \eqref{localisation for bounded b} with $\varepsilon_1$ and $\varepsilon_2$ as stated in the statement of the lemma.  And this completes the proof of the lemma.
\end{proof}


\begin{corollary}
Assume conditions \textnormal{(H1)--(H3)} and that $0<b_*\le b(t)\le b^*<+\infty$ a.e.\ on $(0,1]$.
For every $\rho>0$ and every $u\in\widehat V_\rho$ one has
\[
\|u\|_\infty \le B^*_{\infty,\rho}
\quad\text{where}\quad
B^*_{\infty,\rho}
:=\min\Bigl\{\,B^{*\,\mathrm{Lux}}_{\infty,\rho},\;
M_{\mathrm{sup},\rho}^* \Bigr\}.
\]
Here $\varepsilon_2$ is as in Lemma~\ref{lem:qfree_uniform_b_bounds}, and
\[
B^{*\,\mathrm{Lux}}_{\infty,\rho}
= \frac{1}{\eta_0(\beta-\alpha)^{1/p^-}}
\left[\left(\frac{\rho}{b_*}\right)^{1/p^-}
      + \varepsilon_2\left(\frac{\rho}{b_*}\right)\right],
\qquad
M_{\mathrm{sup},\rho}^*
= C_0^{-1}\,2^{\frac{p^+-p^-}{p^-}}\Bigl(\rho^{\frac{1}{p^-}}+1\Bigr).
\]
\end{corollary}

\begin{proof}
Let $u\in\widehat V_\rho$ so that $\displaystyle\left(b*u^{p(\cdot)}\right)(1) <\rho$. Applying Lemma~\ref{lem:qfree_uniform_b_bounds} with $\displaystyle\rho=\left(b*u^{p(\cdot)}\right)(1)$ and then using monotonicity in $\rho$ gives
\begin{equation}\label{987}
\|u\|_{L^{p(\cdot)}}\le
\left(\frac{\rho}{b_*}\right)^{\!\frac{1}{p^-}}
      +\varepsilon_2\left(\frac{\rho}{b_*}\right).
\end{equation}
By Lemma~\ref{lem:Lp_lower_bound} we have
\begin{equation}\label{987a}
\|u\|_\infty \le \frac{1}{\eta_0(\beta-\alpha)^{1/p^-}}\|u\|_{L^{p(\cdot)}}.
\end{equation}
Combining inequality \eqref{987} with inequality \eqref{987a} yields
\[
\|u\|_\infty
\le
\frac{1}{\eta_0(\beta-\alpha)^{1/p^-}}
\left[\left(\frac{\rho}{b_*}\right)^{1/p^-}
      +\varepsilon_2\left(\frac{\rho}{b_*}\right)\right].
\]

On the other hand, \cite[Lemma 2.9]{goodrich2025unified} yields the upper bound
\(\|u\|_\infty \le M_{\mathrm{sup},\rho}^*\).
Taking the minimum of these two valid bounds gives the claim.
\end{proof}

\begin{remark}
In the small–$\rho$ regime (with $b_*>0$ fixed), the Luxemburg–driven term
\[
\frac{1}{\eta_0(\beta-\alpha)^{1/p^-}}\left[\left(\frac{\rho}{b_*}\right)^{1/p^-}+\varepsilon_2\left(\frac{\rho}{b_*}\right)\right]
\]
tends to $0$, whereas $M_{\mathrm{sup},\rho}^*$ tends to the positive constant $\displaystyle C_0^{-1}2^{\frac{p^+-p^-}{p^-}}$. Hence, for all sufficiently small $\rho$ the new $q$–free cap of $B^*_{\infty,\rho}$ strictly improves the old sup–norm estimate of $M_{\mathrm{sup},\rho}^*$.
\end{remark}

We conclude the paper with an example to illustrate explicitly the significant advantages of the approach we have developed in this paper.  We emphasise that, as Example \ref{example2.12} demonstrates, the improvements are not limited to merely the localisation of the solution via the Luxemburg norm.  There is also a significant improvement in the restrictions on both $\lambda$ and $f$ -- cf., \eqref{eq2.34ggg}--\eqref{eq2.35ggg} versus \eqref{eq2.36ggg}--\eqref{eq2.37ggg} in the example.  We mentioned this earlier in Section 1.

\begin{example}\label{example2.12}
\noindent
We illustrate the effect of our framework by analysing a single problem from \cite[Example 2.13]{goodrich2025unified}, in which $b\equiv\bm{1}$.
Set
\[A(t):=\frac{1000}{3}t\sin{\left(\frac{\pi}{6}t\right)}\] and
\[p(t):=\frac{7}{2}+\frac{3}{2}\cos{t}.\]
With the Dirichlet Green’s function $G$ from \eqref{eq:G_dirichlet}, consider the nonlocal boundary value problem
\begin{equation}\label{eq2.35}
\begin{aligned}
- A\!\left(\int_0^1\big(u(s)\big)^{\frac{7}{2}+\frac{3}{2}\cos{s}}\,ds\right) u''(t) &= \lambda f\big(t,u(t)\big), \quad 0<t<1,\\
u(0)&=u(1)=0.
\end{aligned}
\end{equation}
For this exponent function $p$, we may take $p^-=2$ and $p^+=5$.  Choose $\displaystyle\alpha=\frac{1}{4}$ and $\displaystyle\beta=\frac{3}{4}$; then (see \cite{erbe_positive_2004})
\[
\eta_0=\dfrac{1}{4},\qquad
\max_{\tau\in[0,1]}\int_{\alpha}^{\beta} G(\tau,s)\,ds=\dfrac{3}{32},\qquad
\max_{\tau\in[0,1]}\int_{0}^{1} G(\tau,s)\,ds=\dfrac{1}{8}.
\]
We fix \(\rho_1=\dfrac{1}{2500}\) and \(\rho_2=3\).

Our goal is to compare the localisation provided by \cite{goodrich2025unified} to the localisation we have developed in this work.  In particular, we will show that the localisation induced by the Luxemburg norm is much tighter.  We will also show that the auxiliary conditions imposed on both $\lambda$ and $f$ are much weaker.

\subsection*{Original, sup-norm framework \cite[Example 2.13]{goodrich2025unified}}
Provided that $\lambda$ and $f$ satisfy (to three decimal places of accuracy) both
\begin{equation}\label{eq2.34ggg}
\lambda f^m_{\left[\frac{1}{4},\frac{3}{4}\right] \times \left[\frac{1}{200},\frac{102}{25}\sqrt{2}\right]}>\left(\frac{256}{375}\sqrt{\frac{379}{3}}\right)\sin{\frac{\pi}{15000}}\approx0.002
\end{equation}
and
\begin{equation}\label{eq2.35ggg}
\lambda f^M_{\left[0, 1\right]\times\left[0, 4\sqrt2(\sqrt3+1)\right]}<8000\sqrt[5]{3}\approx9965.848,
\end{equation}
then \cite[Corollary 2.10]{goodrich2025unified} implies that problem \eqref{eq2.35} has at least one positive solution, $u_0$, satisfying
\begin{equation}
u_0\in\widehat{V}_3\setminus\overline{\widehat{V}_{\frac{1}{2500}}}.\notag
\end{equation}
Moreover, we can also conclude from \cite[Corollary 2.10]{goodrich2025unified} that $u_0$ satisfies the localisation
\begin{equation}
0.02=\frac{1}{50}<\Vert u_0\Vert_{\infty}<4\sqrt{2}\big(\sqrt{3}+1\big)\approx15.455.\notag
\end{equation}

\subsection*{Hybrid cone framework (i.e., Theorem \ref{theorem2.25})}
Provided that $\lambda$ and $f$ satisfy (to three decimal places of accuracy) both
\begin{equation}\label{eq2.36ggg}
\lambda f^m_{\left[\frac{1}{4},\frac{3}{4}\right] \times \left[\frac{1}{200},4\sqrt{2}(1/2500)^{1/5}\approx 1.183\right]}>\frac{64}{15}\sin{\left(\frac{\pi}{15000}\left(\frac{1}{2500}\right)^{1/5}\right)}\approx1.869\times10^{-4}
\end{equation}
and
\begin{equation}\label{eq2.37ggg}
\lambda f^M_{[0,1]\times\left[0,\, 4\sqrt{6}\approx 9.798\right]}< \tfrac{1000\cdot 3^{1/5}}{\left\|\int_0^1 G(t,s)\,ds\right\|_{L^{p(\cdot)}}}\approx 1.21\times 10^4,
\end{equation}
then Theorem \ref{theorem2.25} implies that problem \eqref{eq2.35} has at least one positive solution, $u_0$, satisfying
\begin{equation}
u_0\in\widehat{V}_3\setminus\overline{\widehat{V}_{\frac{1}{2500}}}.\notag
\end{equation}
Moreover, we can also conclude from Theorem \ref{theorem2.25} that $u_0$ satisfies the localisation
\begin{equation}
0.02=\frac{1}{50}\le\Vert u_0\Vert_{L^{p(\cdot)}}\le\sqrt{3} \approx 1.732.\notag
\end{equation}

In light of the preceding, here is an enumeration of the improvements the methodology in the present work affords.  Note that as part of this discussion we reference Figures \ref{fig1}--\ref{fig2}.  We also use the terminology (\textquotedblleft inner radius\textquotedblright, \textquotedblleft inner height\textquotedblright, etc.) introduced in Section 1.
\begin{enumerate}[label=(\roman*),leftmargin=1.2cm]
\item \textbf{Inner radius threshold.}
The new threshold is about \(12\times\) better (approximately \(92\%\) reduction) -- cf., Figure \ref{fig2}.

\item \textbf{Inner rectangle height.}
The \(u\)-range shrinks from \(M^{*}_{\mathrm{sup},\rho_1}\approx 5.770\) to \(B^{*}_{\infty,\rho_1}\approx 1.183\), a reduction of about \(79.5\%\), which makes the condition involving $f^m$ easier to satisfy -- cf., Figure \ref{fig1}.

\item \textbf{Outer radius threshold.}
Comparing the right-hand sides of inequalities \eqref{eq2.35ggg} and \eqref{eq2.37ggg}, there is an increase of approximately $21\%$ -- cf., Figure \ref{fig2}.  Such an increase is favourable because it makes condition \eqref{eq2.37ggg} easier to satisfy relative to \eqref{eq2.35ggg}.

\item \textbf{Outer rectangle height.}
The \(u\)-range shrinks from \(M^{*}_{\mathrm{sup},\rho_2}\approx 15.455\) to \(B^{*}_{\infty,\rho_2}\approx 9.798\), a reduction of about \(36.6\%\), which makes the condition involving $f^M$ easier to satisfy -- cf., Figure \ref{fig1}.

\item \textbf{Direct localisation.}
The outer bound tightens from \(15.455\) to \(1.732\), an approximately $88.8\%$ reduction (about \(9\times\) tighter) -- cf., Figure \ref{fig1}.
\end{enumerate}
This demonstrates that the hybrid framework offers a substantial, order-of-magnitude improvement for this class of problem.  Moreover, it illustrates that the Luxemburg norm $\Vert\cdot\Vert_{L^{p(\cdot)}}$ provides both a more natural and a more \textquotedblleft accurate\textquotedblright\ way to localise the solution of \eqref{eq1.1}.
\end{example}

\begin{remark}\label{remark2.32}
There has been some work on investigating \emph{non}-existence of solution to nonlocal differential equations -- especially as related to the magnitude of the parameter $\lambda$.  Both the first author \cite{goodrich19,goodrich25,goodrich24,goodrich27} and Shibata \cite{shibata1,shibata2,shibata3} have conducted studies along these lines, albeit using very different methodologies.  It would, therefore, be interesting to investigate whether the new methodology we have introduced has any salutary effect on non-existence results.  
\end{remark}




\begin{thebibliography}{99}


\bibitem{alves_existence_2015} C. O. Alves, D.-P. Covei, \emph{Existence of solution for a class of nonlocal elliptic problem via sub-supersolution method}, Nonlinear Anal. Real World Appl. \textbf{23} (2015), 1--8.

\bibitem{ambrosetti1} A. Ambrosetti, D. Arcoya, \emph{Positive solutions of elliptic Kirchhoff equations}, Adv. Nonlinear Stud. \textbf{17} (2017), 3--15.

\bibitem{andreianov_numerical_2023} B. Andreianov, E. H. Quenjel, \emph{On numerical approximation of diffusion problems governed by variable exponent nonlinear elliptic operators}, Vietnam J. Math. \textbf{51} (2023), 213--243.

\bibitem{azzouz_existence_2012} N. Azzouz, A. Bensedik, \emph{Existence results for an elliptic equation of Kirchhoff-type with changing sign data}, Funkcial. Ekvac. \textbf{55} (2012), 55--66.

\bibitem{bellamouchi1} C. Bellamouchi, E. Zaouche, \emph{Positive solutions, positive radial solutions and uniqueness results for some nonlocal elliptic problems}, J. Elliptic Parabol. Equ. \textbf{1} (2024), 279--301.

\bibitem{biagi1} S. Biagi, A. Calamai, G. Infante, \emph{Nonzero positive solutions of elliptic systems with gradient dependence and functional BCs}, Adv. Nonlinear Stud. \textbf{20} (2020), 911--931.

\bibitem{biagi2} S. Biagi, A. Calamai, G. Infante, \emph{Nonzero positive solutions of fractional Laplacian systems with functional terms}, Math. Nachr. \textbf{296} (2023), 102--121.

\bibitem{boulaaras0} S. Boulaaras, \emph{Existence of positive solutions for a new class of Kirchhoff parabolic systems}, Rocky Mountain J. Math. \textbf{50} (2020), 445--454.

\bibitem{boulaaras1} S. Boulaaras, R. Guefaifia, \emph{Existence of positive weak solutions for a class of Kirrchoff elliptic systems with multiple parameters}, Math. Meth. Appl. Sci. \textbf{41} (2018), 5203--5210.

\bibitem{bronzi_regularity_2020} A. Bronzi, E. A. Pimentel, G. C. Rampasso, E. V. Teixeira, \emph{Regularity of solutions to a class of variable-exponent fully nonlinear elliptic equations}, J. Funct. Anal. \textbf{279} (2020), no. 12, 108781, 31 pp.


\bibitem{calamai0} A. Calamai, G. Infante, \emph{An affine Birkhoff-Kellogg type result in cones with applications to functional differential equations}, Math. Methods Appl. Sci. \textbf{46} (2023), 11897--11905.

\bibitem{calamai1} A. Calamai, G. Infante, \emph{On the solvability of parameter-dependent elliptic functional BVPs on annular-like domains}, Discrete Contin. Dyn. Syst. Ser. B \textbf{30} (2025), 4287--4295.

\bibitem{cao_class_2021} X.-F. Cao, B. Ge, B.-L. Zhang, \emph{On a class of $p(x)$-Laplacian equations without any growth and Ambrosetti-Rabinowitz conditions}, Adv. Differential Equations \textbf{26} (2021), 259--280.

\bibitem{chen_mathematical_2006} Y. Chen, S. Levine, M. Rao, \emph{Variable exponent, linear growth in image restoration}, SIAM J. Appl. Math. \textbf{66} (2006), 1383--1406.


\bibitem{cianciaruso_solutions_2017-1} F. Cianciaruso, G. Infante, P. Pietramala, \emph{Solutions of perturbed Hammerstein integral equations with applications}, Nonlinear Anal. Real World Appl. \textbf{33} (2017), 317--347.

\bibitem{chu_positive_2025} S. Chu, X. Hao, \emph{Positive solutions for doubly nonlocal boundary value problems with time-varying convolution coefficients}, J. Fixed Point Theory Appl. \textbf{27} (2025), no. 4, Paper No. 93.

\bibitem{correa1} F. J. S. A. Corr\^{e}a, \emph{On positive solutions of nonlocal and nonvariational elliptic problems}, Nonlinear Anal. \textbf{59} (2004), 1147--1155.

\bibitem{correa2} F. J. S. A. Corr\^{e}a, S. D. B. Menezes, J. Ferreira, \emph{On a class of problems involving a nonlocal operator}, Appl. Math. Comput. \textbf{147} (2004), 475--489.

\bibitem{coscia_holder_2009} A. Coscia, G. Mingione, \emph{H\"{o}lder continuity of the gradient of {$p(x)$}-harmonic mappings}, C. R. Acad. Sci. Paris S\'{e}r. I Math, \textbf{328} (4), (1999), 363--368.

\bibitem{delgado_non-local_2020} M. Delgado, C. Morales-Rodrigo, J. R. Santos J\'{u}nior, A. Su\'{a}rez, \emph{Non-local degenerate diffusion coefficients break down the components of positive solution}, Adv. Nonlinear Stud. \textbf{20} (2020), 19--30.

\bibitem{diening_lebesgue_2011} L. Diening, P. Herjulehto, P. H\"{a}st\"{o}, M. R\r{u}\v{z}i\v{c}ka, \emph{Lebesgue and Sobolev spaces with variable exponents}, Springer, Heidelberg, 2011.

\bibitem{do4} J. M. do \'{O}, S. Lorca, J. S\'{a}nchez, P. Ubilla, \emph{Positive solutions for some nonlocal and nonvariational elliptic systems}, Complex Var. Elliptic Equ. \textbf{61} (2016), 297--314.

\bibitem{edmunds_sobolev_2002} D. E. Edmunds, J. Rákosník, \emph{Sobolev embeddings with variable exponent. II.}, Math. Nachr. \textbf{246/247} (2002), 53--67.

\bibitem{erbe_positive_2004} L. H. Erbe, H. Wang, \emph{On the existence of positive solutions of ordinary differential equations}, Proc. Amer. Math. Soc. \textbf{120} (1994), 743--748.

\bibitem{feyfoss2} K. Fey, M. Foss, \emph{Morrey regularity for almost minimizers of asymptotically convex functionals with nonstandard growth}, Forum Math. \textbf{25} (2013), 887--929.

\bibitem{garcia1} M. Garc\'{i}a-Huidobro, R. Man\'{a}sevich, J. Mawhin, S. Tanaka, \emph{Periodic solutions for nonlinear systems of Ode's with generalized variable exponent operators}, J. Differential Equations \textbf{388} (2024), 34--58.

\bibitem{garcia2} M. Garc\'{i}a-Huidobro, R. Man\'{a}sevich, J. Mawhin, S. Tanaka, \emph{Two point boundary value problems for ordinary differential systems with generalized variable exponent operators}, Nonlinear Anal. Real World Appl. \textbf{81} (2025), Paper No. 104196, 13 pp.

\bibitem{ge_small_2021-2} B. Ge, H.-C. Liu, B.-L. Zhang, \emph{Small perturbations of elliptic problems with variable growth in $\mathbb{R}^{N}$}, Proc. A \textbf{477} (2021), Paper No. 20200867, 23 pp.

\bibitem{goodrich_existence_2010} C. S. Goodrich, \emph{Existence of a positive solution to a class of fractional differential equations}, Appl. Math. Lett. \textbf{23} (2010), 1050--1055.


\bibitem{goodrich_new_2018} C. S. Goodrich, \emph{New Harnack inequalities and existence theorems for radially symmetric solutions of elliptic PDEs with sign changing or vanishing Green's function}, J. Differential Equations \textbf{264} (2018), 236--262.

\bibitem{goodrich_topological_2021-1} C. S. Goodrich, \emph{A topological approach to nonlocal elliptic partial differential equations on an annulus}, Math. Nachr. \textbf{294} (2021), 286--309.


\bibitem{goodrich_topological_2021-2} C. S. Goodrich, \emph{A topological approach to a class of one-dimensional Kirchhoff equations}, Proc. Amer. Math. Soc. Ser. B \textbf{8} (2021), 158--172.



\bibitem{goodrich_nonlocal_2021} C. S. Goodrich, \emph{Nonlocal differential equations with convolution coefficients and applications to fractional calculus}, Adv. Nonlinear Stud. \textbf{21} (2021), 767--787.



\bibitem{goodrich19} C. S. Goodrich, \emph{Nonexistence and parameter range estimates for convolution differential equations}, Proc. Amer. Math. Soc. Ser. B \textbf{9} (2022), 254--265.

\bibitem{goodrich20} C. S. Goodrich, \emph{Nonlocal differential equations with $p$-$q$ growth}, Bull. Lond. Math. Soc. \textbf{55} (2023), 1373--1391.


\bibitem{goodrich_application_2024-1} C. S. Goodrich, \emph{An application of Sobolev's inequality to one-dimensional Kirchhoff equations}, J. Differential Equations \textbf{385} (2024), 463--486.

\bibitem{goodrich24} C. S. Goodrich, \emph{A surprising property of nonlocal operators: the deregularising effect of nonlocal elements in convolution differential equations}, Adv. Nonlinear Stud. \textbf{24} (2024), 805--818.

\bibitem{goodrich25} C. S. Goodrich, \emph{Nonexistence of nontrivial solutions to Kirchhoff-like equations}, Proc. Amer. Math. Soc. Ser. B \textbf{11} (2024), 304--314.

\bibitem{goodrich_px-growth_2025} C. S. Goodrich, \emph{$p(x)$-growth in nonlocal differential equations with convolution coefficients}, Adv. Differential Equations \textbf{30} (2025), 115--140.


\bibitem{goodrich27} C. S. Goodrich, \emph{A topological analysis of $p(x)$-harmonic functionals in one-dimensional nonlocal elliptic equations}, Adv. Nonlinear Stud. \textbf{25} (2025), 822--852.


\bibitem{goodrichlizama2} C. S. Goodrich, C. Lizama, \emph{Positivity, monotonicity, and convexity for convolution operators}, Discrete Contin. Dyn. Syst. Series A. \textbf{40} (2020), 4961--4983.

\bibitem{goodrich_existence_2022} C. S. Goodrich, C. Lizama, \emph{Existence and monotonicity of nonlocal boundary value problems: the one-dimensional case}, Proc. Roy. Soc. Edinburgh Sect. A \textbf{152} (2022), 1--27.

\bibitem{goodrich2025unified} C. S. Goodrich, G. Nakhl, \emph{A unified topological analysis of variable growth Kirchhoff-type equations}, Proc. Roy. Soc. Edinburgh Sect. A, doi: 10.1017/prm.2025.10073.




\bibitem{graef2} J. Graef, S. Heidarkhani, L. Kong, \emph{A variational approach to a Kirchhoff-type problem involving two parameters}, Results. Math. \textbf{63} (2013), 877--889.

\bibitem{hao1} X. Hao, X. Wang, \emph{Positive solutions of parameter-dependent nonlocal differential equations with convolution coefficients}, Appl. Math. Lett. \textbf{153} (2024), Paper No. 109063, 5 pp.


\bibitem{infante0} G. Infante, \emph{Nonzero positive solutions of nonlocal elliptic systems with functional BCs}, J. Elliptic Parabol. Equ. \textbf{5} (2019), 493--505.

\bibitem{infante2} G. Infante, \emph{Eigenvalues of elliptic functional differential systems via a Birkhoff-Kellogg type theorem}, Mathematics (2021), 9, 4.

\bibitem{infante99} G. Infante, \emph{Nontrivial solutions of systems of perturbed Hammerstein integral equations with functional terms}, Mathematics (2021), 9, 330.






\bibitem{infante_existence_2014} G. Infante, P. Pietramala, M. Tenuta, \emph{Existence and localization of positive solutions for a nonlocal BVP arising in chemical reactor theory}, Commun. Nonlinear Sci. Numer. Simul., \textbf{19} (2014), 2245--2251.

\bibitem{khamsi_fixed_2015} M. A. Khamsi, W. M. Kozlowski, \emph{Fixed Point Theory in Modular Function Spaces}, Birkh\"{a}user/Springer, Cham (2015).

\bibitem{lan2} K. Q. Lan, \emph{Equivalence of higher order linear Riemann-Liouville fractional differential and integral equations} Proc. Amer. Math. Soc. \textbf{148} (2020), 5225--5234.

\bibitem{lan_compactness_2020} K. Q. Lan, \emph{Compactness of Riemann-Liouville fractional integral operators}, Electron. J. Qual. Theory Differ. Equ. (2020), Paper No. 84, 15 pp.

\bibitem{li1} F. Li, C. Guan, X. Feng, \emph{Multiple positive radial solutions to some Kirchhoff equations}, J. Math. Anal. Appl. \textbf{440} (2016), 351--368.


\bibitem{podlubny_fractional_1999} I. Podlubny, \emph{Fractional Differential Equations}, Academic Press, New York, 1999.



\bibitem{ragusatachikawa1} M. Ragusa, A. Tachikawa, \emph{Boundary regularity of minimizers of $p(x)$-energy functionals}, Ann. Inst. H. Poincar\'{e} Anal. Non Lin\'{e}aire \textbf{33} (2016), 451--476.

\bibitem{rajagopal1} K. R. Rajagopa, M. R\r{u}\v{z}i\v{c}ka, \emph{Mathematical modeling of electrorheological materials}, Contin. Mech. Thermodyn. \textbf{13} (2001), 59--78.

\bibitem{ruzicka_electrorheological_2000} M. R\r{u}\v{z}i\v{c}ka, \emph{Electrorheological Fluids: Modeling and Mathematical Theory}, Lecture Notes in Mathematics, 1748, Springer-Verlag, Berlin (2000).

\bibitem{junior_positive_2018-1} J. R. Santos J\'{u}nior, G. Siciliano, \emph{Positive solutions for a Kirchhoff problem with a vanishing nonlocal element}, J. Differential Equations \textbf{265} (2018), 2034--2043.

\bibitem{shibata1} T. Shibata, \emph{Global and asymptotic behaviors of bifurcation curves of one-dimensional nonlocal elliptic equations}, J. Math. Anal. Appl. \textbf{516} (2022), Paper No. 126525, 12 pp.

\bibitem{shibata2} T. Shibata, \emph{Asymptotic behavior of solution curves of nonlocal one-dimensional elliptic equations}, Bound. Value Probl. (2022), Paper No. 63, 15 pp.

\bibitem{shibata3} T. Shibata, \emph{Exact solutions and bifurcation curves of nonlocal elliptic equations with convolutional Kirchhoff functions}, Bound. Value Probl. (2024), Paper No. 63, 13 pp.

\bibitem{song1} Q. Song, X. Hao, \emph{Positive solutions for nonlocal differential equations with concave and convex coefficients}, Positivity \textbf{28} (2024), no. 5, Paper No. 68, 25 pp.

\bibitem{stanczy1} R. Sta\'{n}czy, \emph{Nonlocal elliptic equations}, Nonlinear Anal. \textbf{47} (2001), 3579--3584.

\bibitem{tachikawa_singular_2014} A. Tachikawa, \emph{On the singular set of minimizers of $p(x)$-energies}, Calc. Var. Partial Differential Equations \textbf{50} (2014), 145--169.



\bibitem{vetro_variable_2022} C. Vetro, \emph{Variable exponent $p(x)$-Kirchhoff type problem with convection}, J. Math. Anal. Appl. \textbf{506} (2022), Paper No. 125721, 16 pp.

\bibitem{wang1} Y. Wang, F. Wang, Y. An, \emph{Existence and multiplicity of positive solutions for a nonlocal differential equation}, Bound. Value Probl. (2011), 2011:5.

\bibitem{webb0} J. R. L. Webb, \emph{Initial value problems for Caputo fractional equations with singular nonlinearities}, Electron. J. Differential Equations (2019), Paper No. 117, 32 pp.


\bibitem{yanma1} B. Yan, T. Ma, \emph{The existence and multiplicity of positive solutions for a class of nonlocal elliptic problems}, Bound. Value Probl. (2016), 2016:165.


\bibitem{yan1} B. Yan, D. Wang, \emph{The multiplicity of positive solutions for a class of nonlocal elliptic problem}, J. Math. Anal. Appl. \textbf{442} (2016), 72--102.


\bibitem{zhikov1} V. V. Zhikov, \emph{Averaging of functionals of the calculus of variations and elasticity theory}, Izv. Akad. Nauk SSSR Ser. Mat.  \textbf{50} (4) (1986), 675--877.

\end{thebibliography}
\end{document}